\documentclass[a4paper,twoside,preprint]{elsarticle}
\usepackage[utf8]{inputenc}
\usepackage[table,xcdraw]{xcolor}
\usepackage{graphicx}
\usepackage{mathtools}      
\usepackage{a4wide}
\usepackage{amstext, amsmath, amssymb}
\usepackage{latexsym}
\usepackage{booktabs}
\usepackage{verbatim} 
\usepackage{url}
\usepackage{import}
\usepackage{fancyvrb}
\usepackage{stmaryrd}
\usepackage{pdfsync}
\usepackage{algorithm}
\usepackage{pgfplotstable}
\pgfplotsset{compat=1.17}
\usepackage{tikz}
\usepackage{hyperref}
\usetikzlibrary{shapes,arrows,arrows.meta,calendar,matrix,backgrounds,folding,calc,positioning,spy,pgfplots.groupplots}

\usepackage{subcaption}
\usepackage[nice]{nicefrac}
\usepackage{tabularx}
\usepackage[color=cyan!30]{todonotes}
\usepackage{cleveref}

\newcommand{\externaldirectory}{external/}
\usepackage{shellesc}
\usetikzlibrary{external}
\tikzexternaldisable

\biboptions{sort&compress}

\newcommand{\st}{\;|\;}

\newcommand{\RR}{\mathbb{R}}

\newcommand{\PP}{\mathbb{P}}
\newcommand{\PPdck}{\mathbb{P}_{\mathrm{dc}}^k(\mcT_h)}
\newcommand{\PPdczero}{\mathbb{P}_{\mathrm{dc}}^0(\mcT_h)}

\newcommand{\PPck}{\mathbb{P}_{\mathrm{c}}^k(\mcT_h)}

\newcommand{\mcP}{\mathcal{P}}

\newcommand{\mcK}{\mathcal{K}}
\newcommand{\mcE}{\mathcal{E}}
\newcommand{\mcN}{\mathcal{N}}
\newcommand{\mcF}{\mathcal{F}}
\newcommand{\mcA}{\mathcal{A}}

\newcommand{\mcT}{\mathcal{T}}
\newcommand{\mcH}{\mathcal{H}}
\newcommand{\mcO}{\mathcal{O}}

\newcommand{\tn}{|\mspace{-1mu}|\mspace{-1mu}|}

\newcommand{\jump}[1]{[#1]}

\newcommand{\mrm}[1]{\mathrm{#1}}

\newcommand{\Gammah}{{\Gamma_h}}
\newcommand{\bnabla}{b \cdot \nabla}
\newcommand{\bhnabla}{b_h \cdot \nabla}
\newcommand{\btildenabla}{\widetilde{b}_h\cdot \nabla}
\newcommand{\nablas}{\nabla_{\Gamma}}

\newcommand{\nablash}{\nabla_{\Gamma_h}}

\newcommand{\bhnablash}{b_h \cdot \nablash}

\newcommand{\ns}{n_{\Gamma}}
\newcommand{\nsh}{n_{\Gamma_h}}

\newcommand{\Ps}{{P}_\Gamma}

\newcommand{\Psh}{{P}_{\Gamma_h}}

\newcommand{\foralls}{\forall\,}

\newcommand{\dx}{\,\mathrm{d}x}

\newcommand{\dr}{\,\mathrm{d}r}

\newcommand{\tnup}{\tn_{\mathrm{up}}}
\newcommand{\tnuph}{\tn_{\mathrm{up},h}}

\newcommand{\tnsd}{\tn_{\mathrm{sd}}}
\newcommand{\tnsdh}{\tn_{\mathrm{sd},h}}
\newcommand{\tnsdasth}{\tn_{\mathrm{sd}\ast,h}}
\newcommand{\tnsdast}{\tn_{\mathrm{sd}\ast}}

\newcommand{\btilde}{{\widetilde{b}}}

\newcommand{\onehalf}{\nicefrac{1}{2}}

\DeclareFontFamily{U}{MnSymbolC}{}
\DeclareSymbolFont{MnSyC}{U}{MnSymbolC}{m}{n}
\DeclareFontShape{U}{MnSymbolC}{m}{n}{
    <-6>  MnSymbolC5
   <6-7>  MnSymbolC6
   <7-8>  MnSymbolC7
   <8-9>  MnSymbolC8
   <9-10> MnSymbolC9
  <10-12> MnSymbolC10
  <12->   MnSymbolC12}{}
\DeclareMathSymbol{\intprod}{\mathbin}{MnSyC}{'270}

\DeclareMathOperator{\Id}{Id}
\DeclareMathOperator*{\essinf}{ess\, inf}

\DefineVerbatimEnvironment{code}{Verbatim}{frame=single,rulecolor=\color{blue}}

\numberwithin{equation}{section}
\numberwithin{figure}{section}
\numberwithin{table}{section}

\newtheorem{theorem}{Theorem}[section]

\newtheorem{lemma}[theorem]{Lemma}

\newtheorem{corollary}[theorem]{Corollary}

\newdefinition{remark}[theorem]{Remark}
\newdefinition{definition}[theorem]{Definition}
\newproof{proof}{Proof}

\numberwithin{equation}{section}

\journal{}

\begin{document}

\begin{frontmatter}
      \title{Stabilized cut discontinuous Galerkin methods
      for advection-reaction problems on surfaces}
    
    \author[ntnu]{Tale Bakken Ulfsby}
    \ead{talebu@math.uio.no}
    
    \author[ntnu,umu]{Andr\'e Massing\corref{cor1}}
    \ead{andre.massing@ntnu.no}
    \cortext[cor1]{Corresponding author}
    
    \author[uppsala,umu]{Simon Sticko}
    \ead{simon@sticko.se}
    
    \address[ntnu]{Department of Mathematical Sciences, Norwegian University of Science and Technology, NO-7491 Trondheim, Norway.
    }
    \address[umu]{Department of Mathematics and Mathematical Statistics, Ume{\aa} University, SE-90187 Ume{\aa}, Sweden}
    \address[uppsala]{Department of Information Technology, Uppsala University, Box 337, 751\,05 Uppsala, Sweden}

    \begin{abstract} 
    We develop a novel cut discontinuous Galerkin (CutDG) method for 
    stationary advection-reaction problems on surfaces embedded in $\RR^d$.
    The CutDG method is based on 
    embedding the surface into a full-dimensional background mesh
    and using the associated discontinuous piecewise polynomials of order $k$
    as test and trial functions.
    As the surface can cut through the mesh in an arbitrary fashion,
    we design a suitable stabilization that enables us to 
    establish inf-sup stability, a priori error estimates, and condition number
    estimates using an augmented streamline-diffusion norm.
    The resulting CutDG formulation is geometrically robust 
    in the sense that all derived theoretical results hold with
    constants independent of any particular cut configuration.
    Numerical examples support our theoretical findings.
    \end{abstract}
    
\end{frontmatter}

\section{Introduction}
\subsection{Background and earlier work}
\label{ssec:background}
Advection-dominated transport processes on surfaces appear in many
important phenomena in science and engineering.  Prominent
applications include flow and transport problems in porous media when
large-scale fracture networks are modeled as composed 2D surfaces
embedded into a 3D bulk
domain~\cite{AlboinJaffreRobertsEtAl2002,AdlerThovertMourzenko2012,Fumagalli2012,BurmanHansboLarsonEtAl2018a}.
Another important instance arises when modeling incompressible
multi-phase flow problems with surfactants~\cite{GanesanTobiska2009,
GrossReusken2011,MuradogluTryggvason2008,GrosReusken2013}, where
potentially low surface diffusion coefficients lead to large surface
Péclet numbers~\cite{AgrawalNeuman1988} in the surface-bounded
surfactants transport model.
Numerical methods for these applications must not only remain stable
and accurate when solving the underlying partial differential
equations (PDEs) in the advection-dominant regime, but preferably
should also be able to handle complicated and evolving surface
geometries with ease. 
As a potential remedy, unfitted finite element methods known as
\emph{cut finite element methods}
(CutFEM)~\cite{BurmanClausHansboEtAl2014,BordasBurmanLarsonEtAl2018} or
\emph{TraceFEM}s~\cite{OlshanskiiReuskenGrande2009} have been
developed for the last 13 years which allow for more flexible handling
of surface geometries by embedding them into a structured and
easy-to-generate background mesh which does not fit the surface
geometry. 
For the development of more classical \emph{fitted} Surface Finite Element Methods (SFEM)
initiated in the seminal work~\cite{Dziuk1988},
we refer to the excellent and
comprehensive reviews~\cite{DziukElliott2013,AndreaBonito2020}.

Using continuous piecewise linear
finite element functions from the ambient space,
the first \textit{unfitted} finite element method for elliptic problems on
surfaces was proposed in~\cite{OlshanskiiReuskenGrande2009},
and later extended to higher-order elements in~\cite{GrandeReusken2016,Reusken2014}. 
As the embedded surface geometry can cut through the background mesh in an arbitrary fashion,
one main challenge in devising unfitted finite element methods is to
ensure their \emph{geometrical robustness} in the sense that they satisfy
similar stability, a priori error, and conditioning number estimates as
their fitted mesh counterparts, but with constants that are
independent of the particular cut configuration.  
A rather universal approach to achieving geometrical robustness is to
augment the weak formulation under consideration with suitably
designed stabilizations also known as \emph{ghost
penalties}~\cite{BurmanClausHansboEtAl2014}.
For Laplace-Beltrami-type problems on surfaces,
ghost penalties based on face stabilization and artificial diffusion
were introduced in \cite{BurmanHansboLarson2015} and \cite{BurmanHansboLarsonEtAl2016c},
respectively.
The contributions from
\cite{GrandeLehrenfeldReusken2018,BuermanHansboLarsonEtAl2018}
then proposed an abstract stabilization CutFEM framework to discretize
elliptic problems using continuous higher-order elements as well as on
embedded manifolds of co-dimension larger than one.
In particular, the volume normal gradient stabilization introduced
in~\cite{GrandeLehrenfeldReusken2018,BuermanHansboLarsonEtAl2018}
was then successfully used to weakly enforce the tangential condition 
in vector-valued problems including the surface
Darcy equation~\cite{HansboLarsonMassing2017}
and the surface Stokes equation
\cite{OlshanskiiQuainiReuskenEtAl2018,OlshanskiiReuskenZhiliakov2021},
all resting upon continuous finite elements.

So far, most fitted and unfitted finite element schemes for surface
PDEs have been designed for diffusion-dominated elliptic
or parabolic type
problems~\cite{DziukElliott2007a,ElliottStinnerStylesEtAl2010,OlshanskiiReusken2014,OlshanskiiReuskenXu2014a,LehrenfeldOlshanskiiXu2018,Zahedi2017,Kovacs2017},
in contrast to the plethora of both stabilized continuous and
discontinuous Galerkin schemes for advection-dominated problems posed
in the Euclidian flat
case, see for instance the comprehensive monograph~\cite{RoosStynesTobiska2008} or
the recent textbook~\cite{ErnGuermond2021c}.  
Interestingly, relevant work on advection-dominated surface problems
appeared first in the context of unfitted finite
elements, starting with~\cite{OlshanskiiReuskenXu2014}, where the
classical Streamline Upwind Petrov--Galerkin (SUPG) approach was
combined with TraceFEM. Later~\cite{HansboLarsonZahedi2015b}
considered a characteristic CutFEM for convection-diffusion problems
on time-dependent surfaces.  Moreover, CutFEM formulations for advection-dominated
problems on surfaces have been proposed using the continuous interior
penalty method~\cite{BurmanHansboLarsonEtAl2018b}, an artificial
diffusion/full-gradient approach~\cite{BurmanHansboLarsonEtAl2018a},
and a normal-gradient stabilized streamline-line diffusion
approach~\cite{BurmanHansboLarsonEtAl2020}.
Finally, an adaptive TraceFEM formulation with mesh adaption guided by
a~posteriori error estimators was developed
in~\cite{ChernyshenkoOlshanskii2015} to solve potentially
advection-dominated advection-diffusion-reaction problems. 
Regarding fitted mesh-based
approaches on explicitly triangulated surfaces,  
variants employing local projection
stabilization~\cite{Simon2017,SimonTobiska2019} and Petrov--Galerkin
type techniques~\cite{BachiniFarthingPutti2021,ZhaoXiaoZhaoEtAl2020}
can be found in the literature.

The development of discontinuous Galerkin (DG) methods
for hyperbolic and advection-dominated problems was initiated 
\cite{Reed1973}, with the first theoretical analyses being presented
in~\cite{LesaintRaviart1974,JohnsonNaevertPitkaeranta1984}.
Later, \cite{BrezziMariniSueli2004} reformulated and generalized the upwind
flux strategy in DG methods by introducing a tunable stabilization parameter.
The advantageous conservation
and stability properties, the high locality, and the
naturally inherited upwind flux term in the bilinear form make DG
methods popular to handle specifically advection-dominated
problems~\cite{Cockburn1999,HoustonSchwabSueli2002,Zarin2005} as
well as elliptic ones~\cite{ArnoldBrezziCockburnEtAl2002,BrezziManziniMariniPietraRusso2000}.
Detailed overviews are provided by the 
monographs~\cite{DiPietroErn2012,HesthavenWarburton2007}.
In contrast, the development of DG methods for advection-dominated problems on surfaces
has been almost completely neglected. Only
the unpublished preprint~\cite{DednerMadhavan2015} proposes
a DG formulation for advection-dominated problems on surfaces
using piecewise linear elements on fitted meshes,
but the presented formulation contains
a geometrically inconsistent velocity-related term 
leading to suboptimal error estimates.
To the best of our knowledge, mostly elliptic problems have been considered
in the context of DG methods,
see, e.g., \cite{DednerMadhavanStinner2013,AntoniettiDednerMadhavanEtAl2015}
and \cite{CockburnDemlow2016} 
for respectively primal and mixed formulations of the Poisson surface problem
on fitted meshes,
while~\cite{BurmanHansboLarsonEtAl2016a} proposed a
stabilized unfitted \emph{cut discontinuous Galerkin method} (CutDG)
based on first-order elements and symmetric interior penalties.  
The latter was then combined in~\cite{Massing2017,LarsonZahedi2021} with a CutDG method
for bulk problems to discretize elliptic bulk-surface problems.
The general stabilization approach is in contrast to alternative 
unfitted discontinuous Galerkin methods for bulk PDEs~\cite{BastianEngwer2009,BastianEngwerFahlkeEtAl2011,SollieBokhoveVegt2011,HeimannEngwerIppischEtAl2013,Saye2017,Saye2017a,MuellerKraemer-EisKummerEtAl2016,KrauseKummer2017},
where troublesome small cut elements are merged with neighbor
elements with a large intersection support by simply extending the
local shape functions from the large element to the small cut element.
While the cell-merging approach automatically upholds local
conservation properties of the original scheme, some drawbacks exist
including the almost complete absence of numerical analysis except
for~\cite{Massjung2012,JohanssonLarson2013}, 
and, most importantly for the present contribution,
the lack of natural extension to surface PDEs.
More specifically, unfitted finite element methods for
surface PDEs do not only suffer from the classical small cut element problem,
but more importantly, the linear dependency of local shape functions when
restricted to a lower-dimensional manifold poses the most significant challenge
which cannot be addressed by a purely cell-merging based approach.

\subsection{New contributions and outline of the paper}
\label{ssec:new-contributions}
In this work, we present a new cut discontinuous Galerkin
(CutDG) method for the discretization of stationary advection-reaction problems
on embedded surfaces. This contribution is part of our long-term efforts to
develop a fully-fledged, stabilized cut discontinuous Galerkin
(CutDG) framework for the discretization of complex multi-physics
interface problems initiated
in~\cite{BurmanHansboLarsonEtAl2016a,GuerkanMassing2019,GuerkanStickoMassing2020}.
Our main motivation is that the stabilization approach provides us with
a versatile theoretical and practical road to formulate, analyze and implement
unfitted discontinuous Galerkin methods.
The presented approach draws inspiration from our earlier
contributions~\cite{BurmanHansboLarsonEtAl2016a,GuerkanStickoMassing2020},
but compared to~\cite{BurmanHansboLarsonEtAl2016a}, we shift here our focus 
from pure diffusion problems to advection-reaction problems while
also considering higher-order elements.
In contrast to our work \cite{GuerkanStickoMassing2020} on
CutDG methods for advection-reaction bulk problems,
we need here to develop new stabilization for the surface-bound PDE.
Such a task is not a straightforward extension of our techniques 
developed in~\cite{GuerkanStickoMassing2020}
as additional stability issues arise in the surface case
which are not present in the bulk version. Moreover, we also
provide precise estimates of all geometrical errors
caused by the geometric approximation of the surface.

We start by briefly recalling the 
advection-reaction model problem on surfaces
and the corresponding weak formulations in Section~\ref{sec:model-problem},
followed by a presentation of the proposed CutDG method in Section~\ref{sec:CutDG-intro}.
Our approach departs from an embedding of the surface $\Gamma$ into a higher dimensional 
background mesh ${\mcT}_h$. 
To account for geometrical errors typically occurring
in surface PDE discretizations, we only assume 
that a piecewise polynomial approximation $\Gamma_h$ of order $k_g$
is available so that the errors in position and normal 
are $O(h^{k_g+1})$ and $O(h^{k_g})$, respectively. 
On the discrete surface $\Gamma_h$ we formulate 
a discontinuous Galerkin method which closely resembles
the classical upwind formulation presented in \cite{BrezziMariniSueli2004},
but uses the discrete function spaces stemming from the background mesh.
The resulting formulation is highly ill-posed due to a)
the potential small intersection between mesh elements and surface discretization
and, more importantly, b) the arising linear dependency of local 3D
shape functions when restricted to the 2D surface.
We like to point out that the popular cell-merging approach for unfitted DG methods for bulk problems
does not provide a remedy for b).
Instead, we add a consistent stabilization $s_h$ to the surface bounded bilinear form $a_h$, 
which renders the method \emph{geometrically robust} and
enables us to prove inf-sup stability and optimal convergence for our CutDG
method with respect to a combined stabilized upwind flux/streamline diffusion-type norm
which is independent of the particular cut configuration.
Our stabilization framework works automatically for higher-order
approximation spaces with polynomial orders $k$ and is not limited to
low-order schemes.  After collecting several auxiliary results regarding norms, interpolation
operators, and geometry-related error estimates in
Section~\ref{sec:norms-etc}, 
we provide a detailed motivation and derivation of a suitable stabilization operator for our
CutDG method in Section~\ref{sec:stab-analysis}. 
Extending the approaches
from~\cite{BurmanHansboLarsonEtAl2016a,BuermanHansboLarsonEtAl2018,GuerkanStickoMassing2020},
we prove that a properly scaled normal gradient volume stabilization
together with low-order jump terms gives us control of certain
rescaled upwind and streamline diffusion norms which are evaluated on
the full background mesh. As a result, we can show that our
formulation satisfies a geometrically robust inf-sup condition with
respect to the stabilized streamline diffusion norm.  The subsequent a
priori error analysis in Section~\ref{sec:aprioriest} builds upon the
classical Strang-type lemma approach which decomposed the total error
into a best approximation error, a consistency error caused by the
stabilization, and a geometrical error arising from the surface
approximation. For each contribution, detailed estimates are given.
Afterward, we demonstrate in Section~\ref{sec::condition-number-est}
that thanks to our stabilization, the condition number of the
resulting system matrix scales exactly as the corresponding
fitted DG upwind formulation in the flat Euclidian case.  Finally, we
corroborate our theoretical findings with a series of numerical
experiments in Section~\ref{sec:numerical-results} where we study both
the convergence properties and geometrical robustness of the proposed
CutDG method.

\section{Model problem}
\label{sec:model-problem}
\subsection{Basic notation}
\label{ssec:preliminaries}
In this work, we let $\Gamma$ be a compact, oriented, and smooth
hypersurface without boundary, embedded in $\RR^d$ and equipped with a
smooth normal field $n : \Gamma \rightarrow \RR^d$. Let $\rho$ denote
the signed distance function that measures the distance in the normal
direction from $\Gamma$, defined on a $\delta$-neighborhood
$U_\delta(\Gamma) = \{ x \in \RR^d \mid \text{dist}(x, \Gamma) <
\delta\}$, see Figure \ref{fig:distance_active} (left).
Then it is well-known that 
the closest point projection $p : U_\delta(\Gamma) \rightarrow \Gamma$ implicitly defined by
\begin{equation}
    p(x) = x - \rho(x) n(p(x))
\end{equation}
is well-defined in $U_\delta(\Gamma)$
provided that $\delta <  \kappa^{-1}$, 
where $\kappa = \max_{i=1, \ldots, d-1}\|\kappa_i\|_{L^\infty(\Gamma)}$ is
the maximum of the principal curvatures of $\Gamma$. 
Using the closest point
projection we can define the extension $u^e$ of a function, $u$ defined on $\Gamma$
to the $\delta$-neighborhood $U_\delta(\Gamma)$ by setting
\begin{equation}
    u^e(x) = u(p(x)).
    \label{eq:extension}
\end{equation}
Conversely, a function $w$ defined on a subset $\widetilde{\Gamma} \subset U_\delta(\Gamma)$ can be lifted back to 
$p(\widetilde{\Gamma}) \subset \Gamma$ via
\begin{equation}
    w^l(x) = w(p^{-1}(x)),
\end{equation}
whenever the closest point mapping $p: \widetilde{\Gamma} \to p(\widetilde{\Gamma})$ is bijective.
Then
\begin{equation}
    (w^l(x))^e = w^l(p(x)) = w \circ p^{-1} \circ p (x) = w.
\end{equation}
Furthermore, for a function $u : \Gamma \rightarrow \RR$, we define the \emph{tangential gradient} $\nabla_\Gamma$ on $\Gamma$ by
\begin{equation}
    \nabla_\Gamma u = P_\Gamma \nabla u^e.
\end{equation}
The operator $P_\Gamma = P_\Gamma(x)$ is the orthogonal projection of $\RR^d$ onto the tangent space of $\Gamma$ at $x \in \Gamma$ given by
\begin{align}
    P_\Gamma = I - n_\Gamma \otimes n_\Gamma,
    \label{eq:Ps-def}
\end{align}
where $I$ is the identity matrix. For a vector field $v$ on $\Gamma$, the \textit{tangential divergence} is defined as
\begin{align}
    \nablas \cdot v = \nabla \cdot v - n_\Gamma \cdot \nabla v n_\Gamma.
    \label{eq:nablas-def}
\end{align}

For any sufficient regular subset $U \subseteq \RR^d$
and $ 0 \leqslant m < \infty$, $1 \leqslant q \leqslant \infty$, we denote by
$W^{m,q}(U)$ the standard Sobolev spaces consisting of those
$\RR$-valued functions defined on $U$ which possess $L^q$-integrable
weak derivatives up to order $m$. 
Their associated norms are denoted by $\|\cdot \|_{m,q,U}$.  As usual,
we write $H^m(U) = W^{m,2}(U)$ and $(\cdot,\cdot)_{m,U}$ and
$\|\cdot\|_{m,U}$ for the associated inner product and norm. 
If unmistakable, we occasionally write
$(\cdot,\cdot)_{U}$ and $\|\cdot \|_{U}$ for the inner products and
norms associated with $L^2(U)$, with $U$ being a measurable subset of
$\RR^d$.
Any norm $\|\cdot\|_{\mcP_h}$ used in this work which
involves a collection of geometric entities $\mcP_h$ should be
understood as the broken norm defined by
$\|\cdot\|_{\mcP_h}^2 = \sum_{P\in\mcP_h} \|\cdot\|_P^2$ whenever
$\|\cdot\|_P$ is well-defined, with a similar convention for scalar
products $(\cdot,\cdot)_{\mcP_h}$.  Any set operations
involving $\mcP_h$ are also understood as element-wise operations,
e.g., $ \mcP_h \cap U = \{ P \cap U \st P \in \mcP_h \} $ and $
\partial \mcP_h = \{ \partial P \st P \in \mcP_h \} $ allowing for a
compact short-hand notation such as $ (v,w)_{\mcP_h \cap U} =
\sum_{P\in\mcP_h} (v,w)_{P \cap U} $ and $ \|\cdot\|_{\partial
\mcP_h\cap U} = \sqrt{\sum_{P\in\mcP_h} \|\cdot\|_{\partial P\cap
U}^2}$.  Moreover, for geometric entities $P$ of Hausdorff dimension
$l$, we denote their $l$-dimensional Hausdorff measure by $|P|_l$.
Finally, throughout this work, we use the notation
$a \lesssim b$ for $a\leqslant C b$ for some generic constant $C$
(even for $C=1$) which varies with the context but is always
independent of the mesh size $h$ and the position of $\Gamma$ relative
to the background $\mcT_h$, but may depend on the dimension~$d$,
the polynomial degree of the finite element functions, the shape
regularity of the mesh, and the curvature of $\Gamma$.
The binary relations $\gtrsim$ and $\sim$ are defined analogously.

\subsection{The continuous problem}
\label{ssec:continousproblem}
We consider the following advection-reaction problem on a surface: find $u : \Gamma \rightarrow \RR$ such that
\begin{equation}
    b \cdot \nabla_\Gamma u + c u  = f \text{ on } \Gamma,
    \label{eq:continuous_problem}
\end{equation}
where $b \in [W^{1, \infty}(\Gamma)]^d$ is a given vector field, and $c \in L^\infty(\Gamma)$ and $f \in L^2(\Gamma)$ are given scalar function.
The corresponding weak form is: find $u \in V = \{ v \in L^2(\Gamma) \st b\cdot\nablas v \in L^2(\Gamma)\}$ such that
\begin{equation}
    a(u, v) = l(v) \hspace{5mm} \forall v \in V,
\label{eq:bilinear_exact}
\end{equation}
with the bilinear form $a(\cdot, \cdot)$ and the linear form $l(\cdot)$ being given by
\begin{align}
    a(u, v) &= (b \cdot \nabla_\Gamma u + cu, v)_\Gamma\nonumber, \\
    l(v) &= (f, v)_\Gamma.
\end{align}
Furthermore, to ensure that problem \eqref{eq:continuous_problem} is well-posed, we assume as usual that
\begin{align}
    \essinf_{x \in \Gamma} \big(c(x) - \frac{1}{2}\nablas \cdot b(x)\big) \geqslant c_0 > 0,
    \label{eq:cb_cond}
\end{align}
for some positive constant $c_0$.

\section{Stabilized cut discontinuous Galerkin methods}
\label{sec:CutDG-intro}
\subsection{Computational domains and discrete function spaces}
\label{ssec:dompdomaindiscrete}
Let $\{\widetilde{\mcT}_h\}_{h}$ be a family of quasi-uniform meshes
consisting of shape-regular 
elements $T$ with element diameter $h_T < \delta$ covering the $\delta$
neighborhood $U_\delta(\Gamma)$ of the surface $\Gamma$.
For simplicity, we assume that our mesh consists of either simplicial or 
cubic elements of dimension $d$.
In computations, one
typically does not have an exact representation of the surface $\Gamma$, but
rather an approximation $\Gamma_h$. 
In this work, the discrete surface $\Gamma_h$ is supposed to
satisfy the following assumptions:
\begin{itemize}
    \item $\Gamma_h \subset U_\delta(\Gamma)$ and the closest point mapping $p : \Gamma_h \rightarrow \Gamma$ is a bijection for $0 < h \leqslant h_0$.
    \item The following estimates hold
    \begin{align}
            \|\rho\|_{L^\infty(\Gamma_h)} \lesssim h^{k_g+1}, \qquad \|n^e - n_h\|_{L^\infty(\Gamma_h)} \lesssim h^{k_g}
            \label{eq:discrete_surf_ass}
    \end{align}
    for a positive integer $k_g \geqslant 1$.
\end{itemize}
Typically, the distance function $\rho$ is approximated by its interpolation
$\rho_h = I_h^{k_g} \rho$ into the space of continuous, piecewise polynomials
of order $k_g$ on $\mcT_h$. Then the discrete surface $\Gamma_h$ given as the
zero level set of $\rho_h$ satisfies the assumptions in
equation~\eqref{eq:discrete_surf_ass}.

For a given background mesh $\widetilde{\mcT}_h$ and discrete surface
$\Gamma_h$, the \textit{active mesh} $\mcT_h$ is defined as the collection of
those mesh elements that have a nonempty intersection with the discrete
surface,
\begin{equation}
    \mcT_h = \{T \in \widetilde{\mcT}_h \mid T \cap \Gamma_h \neq \emptyset\},
    \label{eq:mcT-def}
\end{equation}
while the union of all the active elements is denoted by 
\begin{align}
    \mcN_h = \bigcup_{T \in \mcT_h} T.
    \label{eq:Nh-def}
\end{align}
Further, the set on \emph{interior faces} of the active mesh is given by
\begin{equation}
    \mcF_h = \{F = T^+ \cap T^- \mid T^+, T^- \in \mcT_h, \, T^+ \neq T^- \}.
    \label{eq:mcF-def}
\end{equation}
The face normals $n_F^+$ and $n_F^-$ are the unit normal vectors pointing out
of $T^+$ and $T^-$, respectively. The discrete surface $\Gamma_h$ is assumed
to be piecewise smooth on each element, so we have the set of \emph{surface
parts} $K$ and the set of \emph{interior edges} $E$:
\begin{align}
    \mcK_h &= \{K = \Gamma_h \cap \overset{\circ}{T} \mid T \in \mcT_h\} \cup \{K = \Gamma_h \cap F \mid F \in \mcF_h\}\label{def:K_h},\\
    \mcE_h &= \{E = K^+ \cap K^- \mid K^+, K^- \in \mcK_h\}.
\end{align}
Note that the second set in~\eqref{def:K_h} is included to account for potential corner cases 
where parts of of the embedded surface $\Gamma$ 
intersect non-transversally with a mesh facet $F$ so that
$F\cap\Gamma$ has a non-vanishing $d-1$ dimensional Hausdorff measure.
For every interior edge $E$, the two normals $n_E^{\pm}$ are defined as the
unit vector which is tangential to the surface parts $K^{\pm}$,
perpendicular to $E$, and points outwards with respect to $K^{\pm}$.
Note that the two co-normals $n_E^{\pm}$ are not necessarily co-planar,
see~Figure~\ref{fig:distance_active}.
Each surface element
$K$ also has two pointwise defined normals, giving rise to a piecewise
smooth normal field $n_{\Gamma_h}$ for the discrete surface $\Gamma_h$.
As in the continuous case, the discrete tangential projection $\Psh$ and
tangential gradient $\nablash$ are then defined by
\begin{align}
    \Psh = I - \nsh \otimes \nsh, \qquad
    \nablash u = \Psh \nabla u,
\end{align}
whenever $u$ is (weakly) differentiable and defined in a neighborhood of
$\Gamma_h$.
The various geometric quantities introduced above are illustrated in
Figure~\ref{fig:distance_active}.
Finally, we let
\begin{align}
    V_h = \PPdck = \bigoplus_{T\in\mcT_h} \PP^k(T)
\end{align}
be the discrete space of discontinuous, piecewise polynomials of degree $k$ on $\mcT_h$.
\begin{figure}[htb]
    \centering
  \begin{minipage}[t]{0.45\textwidth}
    \vspace{0pt}
    \includegraphics[width=\textwidth]{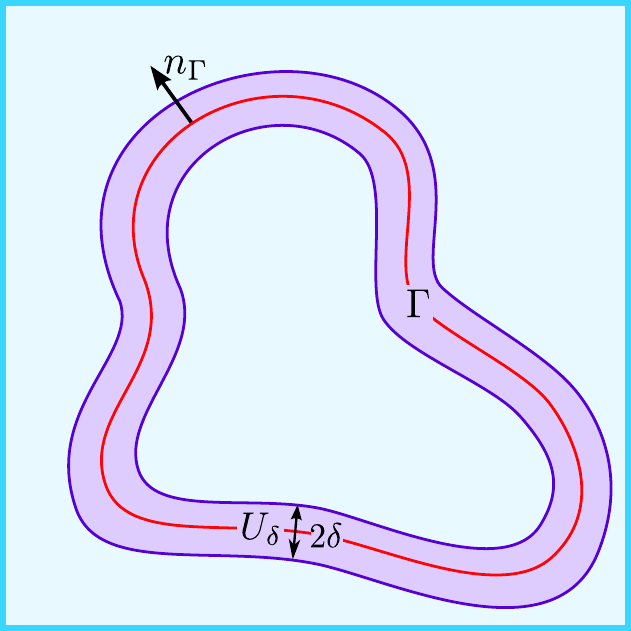}
  \end{minipage}
  \begin{minipage}[t]{0.45\textwidth}
    \vspace{0pt}
    \includegraphics[width=\textwidth]{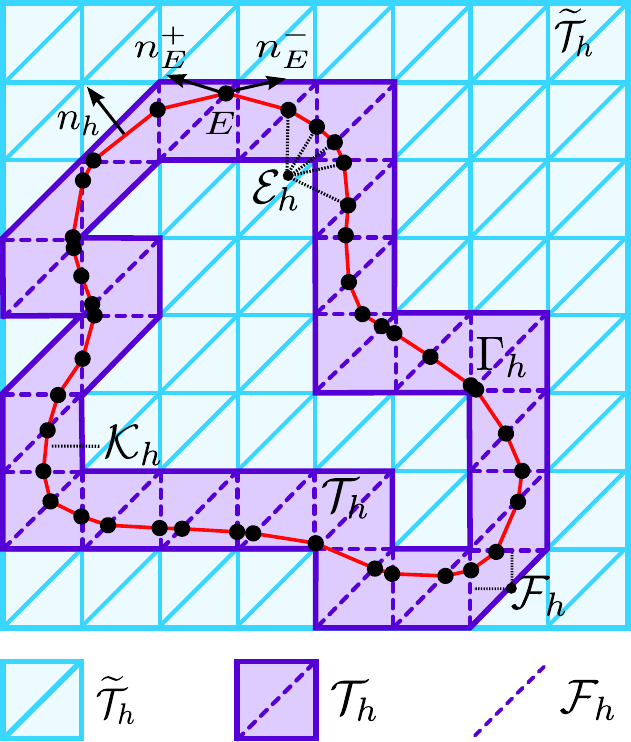}
  \end{minipage}
  \caption{
  Left: The $\delta$ neighborhood of $\Gamma$. Right: Active mesh and discrete surface.
  } \label{fig:distance_active}
\end{figure}

\subsection{Discrete weak formulation}
\label{ssec:disc_weak_form}
To formulate the cut discontinuous Galerkin method for the
advection-reaction problem, we need to define averages and jumps of functions
across edges and faces. For a piecewise discontinuous, possibly
vector-valued function~$\sigma$ defined on the surface part $\mcK_h$, we
define its average and jump over an edge $E \in \mcE_h$ by
\begin{align}
    \{\sigma\}|_E &= \frac{1}{2}(\sigma_E^+ + \sigma_E^-),
    \\
   \jump{\sigma} |_E &= \sigma_E^+ - \sigma_E^-,
\end{align}
respectively.
To account for the fact that the two co-normals $n_E^\pm$ are not
necessarily co-planar, the normal-weighted 
average and jump are given by
respectively
\begin{align}
    \{\sigma ; n_E\}|_E &= \frac{1}{2}(n_E^+ \cdot \sigma^+ - n_E^- \cdot \sigma^-),
    \\
    \jump{\sigma ; n_E}|_E &= (n_E^+ \cdot \sigma^+ + n_E^- \cdot \sigma^-),
\end{align}
which reduces to the known standard definitions in the Euclidean case.
Similarly, for functions~$\mu$ defined on the active background mesh~$\mcT_h$, the average and jump over a 
face~$F \in \mcF_h$ are given by
\begin{align}
    \{\mu\}|_F &= \frac{1}{2}(\mu_F^+ + \mu_F^-),
    \\
   \jump{\mu} |_F &= \mu_F^+ - \mu_F^-.
\end{align}
We can now formulate the cut discontinuous Galerkin based discretization
of the advection-reaction problem~\eqref{eq:continuous_problem}. 
Let $b_h : \Gamma_h \rightarrow
\RR^d$, $c_h: \Gamma_h \rightarrow \RR$ and
$f_h: \Gamma_h \rightarrow \RR$
be suitably defined representations
of $b, c$, and $f$ respectively, defined on the discrete surface $\Gamma_h$. 
Further assumptions for $b_h, c_h$, and $f_h$ are given below and
specific constructions satisfying these assumptions are presented in Section~\ref{ssec:coeff-assump}.
For $v, w \in V_h$, the discrete counterpart of $a(\cdot, \cdot)$ is defined by
\begin{align}
    \label{eq:ahI}
    a_h(v, w) &= (c_h v + b_h \cdot \nablash v , w)_{\mcK_h} 
    -(\{b_h ; n_E\} [v], \{w\})_{\mcE_h}
    + \frac{1}{2}(|\{b_h ; n_E\}|[v], [w])_{\mcE_h}.
\end{align}
However, for a ``naive'' cut discontinuous Galerkin formulation which is solely based on the discrete bilinear~\eqref{eq:ahI} the following issues need to be addressed.
First, as for classical cut finite element formulations of
bulk boundary problems~\cite{BurmanClausHansboEtAl2014}, 
small cut elements with neglegible surface part measure
$|K|_{d-1} \ll h^{d-1}$ and neglegible edge measure $|E|_{d-2} \ll h^{d-2}$
can lead to severely ill-conditioned system matrices.
Second and more importantly, we note that the purely surface-based norms $\| \cdot \|_{\Gamma}$ 
and $\| b\cdot \nablas (\cdot)\|_{\Gamma}$ which are naturally
associated with \eqref{eq:ahI} do not necessarily define proper norms on $V_h$
if the polynomial order $k \geqslant 2$.
For instance, the unit sphere
can be defined by the level set of the second-order polynomial $\phi(x,
    y, z) = x^2 + y^2 + z^2 -1 \in \PP_{\mrm{dc}}^2(\mcT_h)$.  Neglecting
geometric errors and assuming $\Gamma = \Gamma_h$ for a moment, we see
that in that case both $\| \phi \|_{\Gamma}$ and $\| b\cdot \nablas
\phi \|_{\Gamma}$ are zero although $\phi \in V_h$ is clearly
nonvanishing. This issue arises from fact that the aforementioned norms only
account for variations of discrete functions in surface tangential direction but not
for variations in surface normal direction.
As a consequence, it is not possible to establish stability estimates
for the bilinear form~\eqref{eq:ahIa} which are not sensitive to the particular cut
configuration.

A major contribution of the present work is to show how both issues
can be addressed simultaneously by adding a suitably designed
stabilization form $s_h$. Thanks to $s_h$, we gain sufficient control
over functions in $V_h$ in an enhanced streamline-diffusion type norm $\tn \cdot
\tnsdh$ and are able to derive \emph{geometrically robust} stability
properties and optimal error and condition number estimates all of
which are independent of the cut configuration.
The stabilization form $s_h$ is assumed to be symmetric and positive
semi-definite and the final stabilized cut discontinuous Galerkin
formulation is to seek $u_h \in V_h$ such that for all $v_h \in V_h$
\begin{align}
    A_h(u_h, v) \coloneqq
    a_h(u_h, v) + s_h(u_h, v)
    = l(v_h) 
    \coloneqq
    (f_h, v_h)_{\mcK_h}.
    \label{eq:Ah-def}
\end{align}
The design of a suitable stabilization $s_h$ will be the main objective of Section~\ref{sec:stab-analysis}.

\section{Norms, approximation properties, and inequalities}
\label{sec:norms-etc}
Before we turn to the derivation of stability and a priori error estimates for
the discrete problem~\eqref{eq:Ah-def} in the next two sections, we first need to
introduce suitable norms and collect several important auxiliary results.

\subsection{Norms}
\label{subsec:relevantnorms}
First, inspired by the theoretical analysis
in~\cite{DiPietroErn2012,GuerkanStickoMassing2020}, we define a characteristic or
reference time $\tau_c$ via
\begin{align}
    \tau_c^{-1} = \|c\|_{0, \infty, \Gamma} + |b|_{1, \infty, \Gamma} + b_{\infty} \kappa,
    \label{eq:def-t_c}
\end{align}
where $b_{\infty} = \|b\|_{0, \infty,
\Gamma}$ denotes the reference velocity and $\kappa$ is the maximum principal curvature of $\Gamma$ defined in
Section~\ref{sec:model-problem}.
Throughout this work, we assume that the mesh is sufficiently fine in the sense
that
\begin{align}
    h \leqslant b_{\infty} \tau_c 
    \quad
    \Leftrightarrow 
    \quad
    \tau_c^{-1}\phi_b \leqslant 1,
    \label{eq:mesh-resol-assump}
\end{align}
introducing the scaling factor
\begin{align}
    \phi_b = h / b_{\infty},
    \label{eq:def-phi_b}
\end{align}
which will be omnipresent in the forthcoming stability and error analysis.
Assumptions~\eqref{eq:mesh-resol-assump} ensure
that the individual inequalities
\begin{align}
\|c\|_{0, \infty, \Gamma} h \leqslant b_{\infty},
\quad 
|b|_{1, \infty, \Gamma} \leqslant \dfrac{b_{\infty}}{h},
\quad \text{and} \quad
h \leqslant \dfrac{1}{\kappa}
    \label{eq:c_phi_2}
\end{align}
are all satisfied. The first inequality simply means that on an
element level, problem~\eqref{eq:continuous_problem} can be considered
advection-dominant, while the second one ensures that the velocity
field~$b$ is sufficiently resolved. The third inequality in
\eqref{eq:c_phi_2} is just a reformulation of our previous assumption
that the active mesh lies within an $\delta$-neighborhood
$U_{\delta}(\Gamma)$ for which the closest point projection is
uniquely defined, cf.~Section~\ref{ssec:dompdomaindiscrete}.

Next, we define the upwind and the scaled streamline diffusion norm by
\begin{align}
    \tn v \tnup^2 &= \tau_c^{-1} \|v\|_{\mcK_h}^2 + \frac{1}{2} \| |\{b_h; n_E\}|^{\onehalf} [v]\|_{\mcE_h}^2,
    \label{eq:upw-norm-def}
    \\
    \tn v \tnsd^2 &= \tn v \tnup^2 + \|\phi_b^{\onehalf} b_h \cdot \nablash v\|_{\mcK_h}^2,
    \label{eq:sd-norm-def}
\end{align}

On a few occasions, we will also employ a slightly stronger norm than $\tn \cdot \tnsd$
defined by
\begin{align}
    \tn v \tnsdast^2 
    = \phi_b^{-1}\| v\|_{\mcK_h}^2.
    + \|\phi_b^{\onehalf} b_h \cdot \nablash v\|_{\mcK_h}^2
    + b_{\infty} \|v\|_{\partial \mcK_h}^2 
    \label{eq:sdhast-def}
\end{align}
as it immediately leads to the following useful boundedness results.
\begin{lemma}
    For $v \in H^1(\Gamma) \oplus V_h$ and $w \in V_h$ it holds that
\begin{equation}
    a_h(v,w) \lesssim \tn v \tnsdast \tn w \tnsd.
    \label{eq:A_h_boundedness}
\end{equation}
\end{lemma}
The corresponding stabilized norms
\begin{align}
    \tn v \tn_{\bigstar ,h}^2 = \tn v \tn_{\bigstar}^2 + |v|_{s_h}^2
    \quad \text{for } \bigstar \in \{\mrm{up},\mrm{sd}, \mrm{sd}\ast \}
\end{align}
will play a crucial role in the theoretical analysis of the proposed cut
discontinuous Galerkin method. Here, as usual, $|\cdot|_{s_h}$
refers to the semi-norm induced by the symmetric stabilization 
bilinear form $s_h$.

\subsection{Useful inequalities}
In the forthcoming analysis, we will use 
several inverse inequalities which hold for discrete functions $v_h \in V_h$, namely
\begin{alignat}{5}
    \|D^j v_h\|_T &\lesssim h^{i-j}\|D^i v_h\|_T && \quad \foralls T \in \mcT_h, && \quad 0 \leqslant i \leqslant j,
    \label{eq:inverse}\\
    \|D^j v_h\|_{\partial T} &\lesssim h^{i-j-1/2}\|D^i v_h\|_T &&\quad\foralls T \in \mcT_h, && \quad 0 \leqslant i \leqslant j - 1/2,
    \label{eq:inverseandtrace}\\
    \|D^j v_h\|_{\Gamma \cap T} &\lesssim h^{i-j-1/2}\|D^i v_h\|_T &&\quad \foralls T \in \mcT_h, && \quad 0 \leqslant i \leqslant j - 1/2,
    \label{eq:inversegamma}
    \\
    \|D^j v_h\|_{E \cap F} &\lesssim h^{i-j-1/2}\|D^i v_h\|_F 
    &&  \quad \foralls (E, F) \in \mcE_h \times \mcF_h,
    && \quad 0 \leqslant i \leqslant j - 1/2,
    \label{eq:inverseE}
\end{alignat}
while for functions $v \in H^1(\mcT_h)$ the trace inequalities
\begin{alignat}{3}
    \|v\|_{\partial T} &\lesssim h^{-\onehalf}\|v\|_{T} + h^{\onehalf}\|\nabla v\|_T && \quad \foralls T \in \mcT_h,\\
    \|v\|_{\Gamma \cap T} &\lesssim h^{-\onehalf}\|v\|_{T} + h^{\onehalf}\|\nabla v\|_T && \quad \foralls T \in \mcT_h,
    \label{eq:tracegamma}\\
    \|u\|_{E \cap F} &\lesssim h^{-\onehalf}\|u\|_{F} + h^{\onehalf}\|\nabla u\|_F 
    &&  \quad \foralls (E, F) \in \mcE_h \times \mcF_h,
    \label{eq:traceF}
\end{alignat}
will be extremely useful.  
All the above inequalities are consequences of similar well-known inverse estimates which can be found in, e.g., ~\cite[Sec. 4]{HansboHansboLarson2003}.

\subsection{Quasi-interpolation operators}
\label{ssec:quasi-interpolation}
Next, we define two suitable quasi-interpolation operators which will be heavily used
throughout the stability and a priori error analysis.
First, let $\pi_h^* : L^2(\mcT_h) \rightarrow V_h$ be the standard $L^2$ projection which 
for $v \in H^s(\mcT_h)$ and $r \coloneqq \min\{s, k+1\}$ satisfies the error estimates
\begin{alignat}{5}   
    &\|v - \pi_h^* v\|_{k, T} \lesssim h^{r-k} |v|_{r, T}
    &&\quad \foralls T \in \mcT_h, 
    && \quad 0 \leqslant k \leqslant r,
    \label{eq:L2T}
    \\
    &\|v - \pi_h^* v\|_{k, F} \lesssim h^{r-k-1/2} |v|_{r, F}
    &&\quad \foralls F \in \mcF_h, \label{eq:L2F}
    &&  \quad 0 \leqslant k \leqslant r -1/2,
\end{alignat}
see~\cite[Sec. 1.4.4]{DiPietroErn2012}.
Now define $\pi_h : H^s(\Gamma) \rightarrow V_h$ by taking the $L^2$-projection of the extension of $v$, so that $\pi_h v = \pi_h^* v^e$ for $v \in H^s(\Gamma)$. 
To derive error estimates for this quasi-interpolation operator, we recall the co-area formula
\begin{align}
    \int_{U_\delta} f(x) \dx = \int_{-\delta}^{\delta} \Big(\int_{\Gamma(r)} f(y, r) d\,\Gamma_r(y)\Big) \dr,
\end{align}
which can be found for instance in~\cite[Thm. 3.11]{EvansGariepy2015}.
Thanks to the co-area formula and the assumption $\mcT_h \subset U_\delta(\Gamma)$, the extension operator satisfies the estimate
\begin{align}
    \|v^e\|_{k, U_{\delta}(\Gamma)} \lesssim \delta^{\onehalf} \|v\|_{k, \Gamma}, \hspace{5mm} 0 \leqslant k \leqslant s,
    \label{eq:extensionstability}
\end{align}
for $0 < \delta < \delta_0$, where $\delta \sim h$.

For the forthcoming design and analysis of the stabilization $s_h(\cdot, \cdot)$,
we need to review some basic facts about the Oswald interpolation operator 
$\mcO_h : \PPdck \rightarrow \PPck$, which maps discontinuous piecewise polynomials on $\mcT_h$
to continuous ones. 
For a function $v_h \in
\PPdck$, its continuous version $\mcO_h(v_h)$ is defined in each interpolation
node $x_i$ by taking the average
\begin{align}
    \mcO_h(v_h)(x_i) = \frac{1}{\text{card}(\mcT_h(x_i))} \sum_{T \in \mcT_h(x_i)} v_h|_T (x_i),
\end{align}
where $\mcT_h(x_i)$ is the set of all elements $T \in \mcT_h$ sharing the node
$x_i$. The deviation of $\mcO_h(v_h)$ from $v_h$ can then be measured by the jumps of $v_h$ across faces
as stated in the following lemma. A proof can be found in~\cite[Lem. 3.2]{BurmanErn2007}.
\begin{lemma}
For $v_h \in \PPdck$ we have
\begin{align}
    \|v_h - \mcO_h(v_h)\|_T^2 \lesssim \sum_{F \in \mcF_h(T)} h \|[v_h]\|_F^2,
    \label{eq:oswald_interp_est}
\end{align}
where $\mcF_h(T)$ denotes all faces $F$ in $\mcF_h$ that intersects $T$.
\label{lem:oswald}
\end{lemma}

\subsection{Domain perturbation related estimates}
Using the definition of the discrete surface gradient $\nablash : V_h
\rightarrow \RR^d$ and applying the chain rule, we have the well-known identity
\begin{align}
    \nablash u^e = B^T \nablas u,
    \label{eq:nablash-ue}
\end{align}
where $B = P_\Gamma (I - \rho \mcH ) P_{\Gamma_h} : T_x(K) \rightarrow T_{p(x)} (\Gamma)$, and $\mcH = \nabla \otimes \nabla \rho$. 
Note that for $h$ small enough, the linear mapping $\Ps \Psh: T_x(K) \rightarrow T_{p(x)} (\Gamma)$ 
and thus $B = \Ps(I-\rho \mcH) \Psh  = \Ps  \Psh + h^{k_g+1}$ are invertible 
as mappings from the discrete to the continuous tangential space,
thanks to geometry assumption~\eqref{eq:discrete_surf_ass}. 
Using~\eqref{eq:nablash-ue}, we can also write the lifting of the gradient from $\Gamma_h$ to $\Gamma$ by using
\begin{align}
    \nablash w = \nablash (w^l)^e = B^T \nablas w^l,
\end{align}
so
\begin{align}
    \nablas w^l = B^{-T}\nablash w.
\end{align}
The measure on $\Gamma$ can be expressed as
\begin{align}
    d\, \Gamma = |B| d\, \Gamma_h,
\end{align}
where $|B|$ is the absolute value of the determinant of $B$. For $B$ and $|B|$ we recall
that the assumptions made in~\eqref{eq:discrete_surf_ass}
imply the following estimates
\begin{align}
    \|B\|_{L^{\infty}(\Gamma)} \lesssim 1, \quad \|B^{-1}\|_{L^{\infty}(\Gamma_h)} \lesssim 1, \quad \|P_{\Gammah}P_\Gamma - B^{-1}\|_{L^{\infty}(\Gammah)} \lesssim h^{k_g +1},
    \label{eq:B_estimates}
\end{align}
and
\begin{align}
    \|1 - |B|\|_{L^\infty(\Gammah)} \lesssim h^{k_g+1}, \quad \||B|\|_{L^\infty(\Gammah)} \lesssim 1, \quad \||B^{-1}|\|_{L^\infty(\Gammah)} \lesssim 1,
    \label{eq:B_det_estimates}
\end{align}
and we refer to \cite{} for the details.
This leads to the norm equivalences
\begin{align}
    \|v^l\|_{L^2(\Gamma)} &\sim \|v\|_{L^2(\Gamma_h)},\label{eq:normeql2}\\
     \|\nablas v^l\|_{L^2(\Gamma)} &\sim \|\nablash v\|_{L^2(\Gamma_h)}\label{eq:normeqh1}
\end{align}
for $v \in H^1(\Gamma)^e \oplus V_h$. Proofs of the above identities, inequalities and norm equivalences can be found in, e.g., \cite{DziukElliott2013,GrandeLehrenfeldReusken2018,BurmanHansboLarsonEtAl2016a,BuermanHansboLarsonEtAl2018}.

\subsection{Assumption on the discrete coefficients}
\label{ssec:coeff-assump}
For the discrete coefficient functions $b_h, c_h$ and $f_h$,
we now formulate several minimal assumptions for the
forthcoming stability and error analysis to hold. 
First, as the expression $\bhnablash v$ only involves tangential components of~$b_h$,
we simply require that the velocity field $b_h$ is purely tangential.
Next, we assume that $b_h$ and $c_h$ admit a discrete version of \eqref{eq:cb_cond}, 
\begin{align}
    \essinf_{x \in \Gamma_h}\big(c_h(x) - \frac{1}{2}\nablash \cdot b_h(x)\big) \geqslant c_{0, h} > 0.
    \label{eq:bh-essinf-assump}
\end{align}
with some positive and $h$-independent constant $c_{h,0}$.
Further, the following approximation properties are supposed to hold,
\begin{align}
    \|P_{\Gamma_h} b^e - b_h\|_{L^\infty(\mcK_h)} \lesssim C_b h^{k_g+1},
    \label{eq:bh_assumption_prel}
    \\ 
    \|c^e - c_h\|_{L^\infty(\mcK_h)} \lesssim C_c h^{k_g+1}, 
    \label{eq:ch_assumption_prel}
    \\
     \|f^e - f_h\|_{L^\infty(\mcK_h)} \lesssim  C_f h^{k_g+1}.
    \label{eq:fh_assumption_prel}
\end{align}
In addition to the $k_g+1$ order estimate~\eqref{eq:bh_assumption_prel}, 
we also assume a first-order estimate of the form
\begin{align}
\|P_{\Gamma_h} b^e - b_h\|_{L^\infty(\mcK_h)}
&\lesssim h (b_{\infty} \kappa + |b|_{1,\infty,\Gamma}),
\label{eq:b-bh_assump}
\end{align}
which we will see is sufficient to ensure that stabilized CutDG formulation~\eqref{eq:Ah-def}
satisfies a discrete inf-sup condition.
Finally, we also assume the existence of a piecewise constant vector field~$\btilde_h$
satisfying
\begin{align}
\|P_{\Gamma_h} b^e - \btilde_h\|_{L^\infty(\mcK_h)}
&\lesssim h (b_{\infty} \kappa + |b|_{1,\infty,\Gamma})
\label{eq:b-bhtilde_assump}
\\
\|b^e - \btilde_h\|_{0, \infty, T} 
&\lesssim h (
    b_{\infty} \kappa + |b|_{1, \infty, \Gamma}),
\quad
\|\btilde_h\|_{0, \infty, T} \lesssim \|b\|_{0, \infty, \Gamma}.
\label{eq:bhassumptions}
\end{align}
Since the extended vector field $b^e$ is in $W^{1, \infty}(U_{\delta}(\Gamma))$, 
such a patch-wise defined, locally constant, vector field $\widetilde{b}_h$ satisfying the
assumptions above can always be constructed, by for example taking the value
of $b_h$ at a point in the patch.

Thanks to the domain-perturbation-related estimates reviewed in the
previous section, the approximation properties can be reformulated in
a manner that will be more convenient in the analysis of the
geometrical errors presented in Section~\ref{ssec:geom-err-est}.
\begin{lemma}
    \label{lem:coeff-est}
    Assume that $b_h, c_h$, and $f_h$ satisfy~\eqref{eq:bh_assumption_prel}--\eqref{eq:fh_assumption_prel},
    then it holds that
    \begin{align}
        \||B| B^{-1} b^e - b_h\|_{L^\infty(\mcK_h)} & \lesssim C_b h^{k_g+1},
        \label{eq:bh_assumption}                                          \\
        \||B|c^e - c_h\|_{L^\infty(\mcK_h)}         & \lesssim  C_c h^{k_g+1},
        \label{eq:ch_assumption}                                          \\
        \||B|f^e - f_h\|_{L^\infty(\mcK_h)}         & \lesssim C_f h^{k_g+1}.
        \label{eq:fh_assumption}
    \end{align}
\end{lemma}
\begin{proof}
    The proof in~\cite[Sec. 4.2]{BurmanHansboLarsonEtAl2020} for $k_g=1$ immediately generalizes
    to our geometrical assumptions, but for the reader's convenience, we   
    provide a short proof in~\ref{sec:appendix}.
\qed\end{proof}

We conclude this section by recalling an estimate for the co-normal jump of the discrete
velocity $b_h$ which will come in handy when turning to the stability and a priori error
analysis of the proposed CutDG method.
\begin{lemma}
    \label{lem:est-bh-jump}
    Assume that the geometric approximation assumption~\eqref{eq:discrete_surf_ass} holds
    and that $b_h$ satisfies~\eqref{eq:bh_assumption_prel}.
    Then 
    \begin{align}
        \|[b_h; n_E]\|_{L^\infty(\mcE_h)}           
        & \lesssim C_b h^{k_g+1}.
    \end{align}
\end{lemma}
\begin{proof}
    A proof for the case $d=2$ and $k_g = 1$ can be found in, e.g., \cite[Lemma 3.6]{OlshanskiiReuskenXu2014}. For the reader's convenience, 
    a slightly generalized proof for $d \geqslant 2$ and $k_g \geqslant 1$ is given in~\ref{sec:appendix}.
\qed\end{proof}

\section{Stability analysis}
\label{sec:stab-analysis}
A key observation made at the end of Section~\ref{ssec:disc_weak_form} 
is that the purely surface-based bilinear form $a_h$ and its
associated ``norm'' $\tn \cdot \tnsd$ do not provide sufficient control over
a discrete function $v_h \in V_h$ defined on the active mesh $\mcT_h$.
The major objective of this section is to show that if we augment
$a_h$ by a suitably constructed stabilization form $s_h$, control over
$V_h$ in the resulting enhanced norm $\tn \cdot \tnsdh$ is regained
which allows us to prove a geometrically robust inf-sup condition with
respect to the $\tn\cdot\tnsdh$ norm.

\subsection{Construction of the stabilization form $s_h$}
\label{ssec:ghost_penalties}
As the model problem~\eqref{eq:continuous_problem} consists of an
advection-reaction operator, it is natural to assume that a suitably
designed stabilization
will acknowledge this.
We thus start by considering the norm part which is typically associated with the reaction term.
Here, it is more natural to use the \emph{rescaled} or \emph{extended} $L^2$
norm $(\tau_c h)^{-\onehalf}\| v_h \|_{\mcT_h}$ instead of the purely surface-based norm
$\tau_c^{-\onehalf} \| \cdot \|_{\Gamma_h}$ since the former provides
a proper norm
for discrete functions defined on the active mesh $\mcT_h$.
The following lemma (first proved in~\cite{BuermanHansboLarsonEtAl2018,GrandeLehrenfeldReusken2018})
then shows that for a \emph{continuous}, discrete function $v_h \in \PPck$,
the extended $L^2$ norm can be bounded by the surface $L^2$ norm if enhanced by the
volume-based normal gradient stabilization, which provides
sufficient control in normal directions.
\begin{lemma}
For $v \in \PPck$ it holds that
\begin{align}
    h^{-1}\|v\|_{\mcT_h}^2 &\lesssim \|v\|_{\mcK_h}^2 + h \|n_\Gammah \cdot \nabla v\|_{\mcT_h}^2.
    \label{eq:ineqcontinuous}
\end{align}
\label{lem:ineqcontinuous}
\end{lemma}
\noindent Our first task is to extend the previous lemma to \emph{discontinuous} element-wise polynomials.
\begin{lemma}
\label{lem:inverse_discontinuous}
For $v \in \PPdck$ it holds that
\begin{align}
    h^{-1}\|v\|_{\mcT_h}^2 &\lesssim \|v\|_{\mcK_h}^2 + \|[v]\|_{\mcF_h}^2 + h \|n_{\Gamma_h} \cdot \nabla v\|_{\mcT_h}^2.
\end{align}
\end{lemma}
\begin{proof} We use the Oswald interpolant $\mcO_h$ reviewed in Section~\ref{ssec:quasi-interpolation} to
create a continuous version of~$v$ which is eligible for an application of
Lemma~\ref{lem:ineqcontinuous}. Set $\widetilde{v} = \mcO_h(v) \in
\PPck$. 
Using Lemma~\ref{lem:ineqcontinuous} in combination
with the inverse estimates~\eqref{eq:inverse},\eqref{eq:inversegamma},
and Lemma~\eqref{lem:oswald} on the Oswald interpolant
results in the following chain of estimates:
\begin{align}
    \|v\|_{\mcT_h}^2 &\lesssim \|\widetilde{v}\|_{\mcT_h}^2 + \|v - \widetilde{v}\|_{\mcT_h}^2\\
    &\lesssim h \|\widetilde{v}\|_{\mcK_h}^2 + h^2 \|n_{\Gamma_h} \cdot \nabla \widetilde{v}\|_{\mcT_h}^2 + \|v - \widetilde{v}\|_{\mcT_h}^2\\
    &\lesssim h \|v\|_{\mcK_h}^2 + h \|\widetilde{v} - v\|_{\mcK_h}^2 + h^2 \|n_{\Gamma_h} \cdot \nabla v\|_{\mcT_h}^2 + h^2 \|n_{\Gamma_h} \cdot \nabla (\widetilde{v} - v)\|_{\mcT_h}^2 + \|v - \widetilde{v}\|_{\mcT_h}^2\\
    &\lesssim h \|v\|_{\mcK_h}^2 + \|\widetilde{v} - v\|_{\mcT_h}^2 + h^2 \|n_{\Gamma_h} \cdot \nabla v\|_{\mcT_h}^2 +  \|\widetilde{v} - v\|_{\mcT_h}^2 + \|v - \widetilde{v}\|_{\mcT_h}^2\\
    &\lesssim h \|v\|_{\mcK_h}^2 + h^2 \|n_{\Gamma_h} \cdot \nabla v\|_{\mcT_h}^2 + \|v - \widetilde v\|_{\mcT_h}^2\\
    &\lesssim h \|v\|_{\mcK_h}^2 + h^2 \|n_{\Gamma_h} \cdot \nabla v\|_{\mcT_h}^2 + h\|[v]\|_{\mcF_h}^2.
\end{align}
\qed\end{proof}
The previous lemma motivates the following definition of a reaction-term associated stabilization~$s_h^c$ of the form
\begin{align}
    \label{eq:def-shc}
    s_h^c(v, w) 
    \coloneqq
     \gamma^c_0 \tau_c^{-1}([v], [w])_{\mcF_h} + \gamma^c_n \tau_c^{-1} h
    (n_{\Gamma_h} \cdot \nabla v, n_{\Gamma_h} \cdot \nabla
    w)_{\mcT_h}
    \quad \text{for } v,w \in V_h,
\end{align}
with $\gamma_0^c$ and $\gamma_n^c$ being dimensionless, positive stability
parameters. 
Thanks to Lemma~\ref{lem:inverse_discontinuous},
incorporating $s_h^c$ into $s_h$ gives us control over the
extended $L^2$ norm in the sense that
\begin{align}
    (\tau_c h)^{-1} \| v \|_{\mcT_h}^2
    \lesssim
    \tn v\tnuph^2
    \label{eq:l2norm-extension-property}
\end{align}
holds for $v \in V_h$. 
We refer to \eqref{eq:l2norm-extension-property} by saying that $s_h^c$
satisfies an \emph{$L^2$-norm extension property}.

We turn to the stabilization of the $\tn \cdot \tnsd$ norm.
As in the analysis of the classical upwind stabilized DG
method~\cite{BrezziMariniSueli2004,DiPietroErn2012}, we will
exploit that the scaled streamline derivative $\phi_b b_h \cdot \nablash
v_h = \phi_b b_h \cdot \nabla v_h$ is (almost) a valid test function if
only $b_h$ is replaced by an element-wise constant, $\widetilde{b}_h$,
which satisfies~\eqref{eq:b-bhtilde_assump}.
Moreover, similar to $\btilde_h$
we can construct 
an element-wise constant approximation, $\widetilde{n}_{\Gamma_h}$,
of $n_{\Gamma_h}$ which
satisfies the estimate
\begin{align}
\|n_{\Gamma_h} -\widetilde{n}_{\Gamma_h}\|_{\infty} \lesssim \kappa h.
\label{eq:nh-nhtilde-est}
\end{align}
The next lemma will help us
to quantify the errors introduced 
when switching between 
$b_h$ and $\widetilde{b}_h$ respectively 
$n_{\Gammah}$ and $\widetilde{n}_{\Gammah}$ 
in the forthcoming stability analysis.
\begin{lemma}
    For $v \in V_h$, it holds that
\begin{align}
    \bigl(b_{\infty} \| \widetilde{n}_{\Gammah}-n_{\Gammah} \|_{0,\infty,\mcT_h}^2 +
    \dfrac{\phi_b}{h} \|\widetilde{b}_h - b_h\|_{0,\infty, \mcT_h}^2 
    \bigr)
    \|\nabla v\|_{\mcT_h}^2 
    \lesssim 
    (\tau_c h)^{-1}\| v\|_{\mcT_h}^2
    \lesssim 
    \tn v \tnuph^2.
\label{eq:bh-b_est_upw}
\end{align}
\label{lem:bh-b_est_upw}
\end{lemma}
\begin{proof}
    First, simply using inverse estimate~\eqref{eq:inverse}
    together with~\eqref{eq:nh-nhtilde-est}
    shows that
    \begin{align}
    b_{\infty} \| \widetilde{n}_{\Gammah}-n_{\Gammah} \|_{0,\infty,\mcT_h}^2
    \|\nabla v\|_{\mcT_h}^2
    &\lesssim 
    (b_{\infty} \kappa) (\kappa h)  h^{-1}\| v\|_{\mcT_h}^2
    \lesssim 
    (\tau_c h)^{-1}\| v\|_{\mcT_h}^2
    \label{eq:bhtildevsbh-I}
    \end{align}
    since $b_{\infty} \kappa \leqslant \tau_c^{-1}$
    by the definition of $\tau_c$ \eqref{eq:def-t_c} and $\kappa h \lesssim 1$
    \eqref{eq:c_phi_2}.
    Next, by \eqref{eq:def-t_c} and assumption~\eqref{eq:mesh-resol-assump}
    $(|b|_{1,\infty, \Gamma} + b_{\infty} \kappa) \phi_b \leqslant
    \tau_c^{-1}\phi_b \lesssim 1$ 
    holds
    and therefore adding and subtracting $\Psh b$
    in the second term in the left-hand side of~\eqref{eq:bh-b_est_upw}
    together with \eqref{eq:b-bh_assump}, \eqref{eq:b-bhtilde_assump}, and~\eqref{eq:inverse}
    yields
\begin{align}
  \dfrac{\phi_{b}}{h} \|\btilde_h - b_h\|_{0,\infty,\mcT_h}^2 \|\nabla v \|_{\mcT_h}^2
  \lesssim
  \dfrac{\phi_{b}}{h}
  (|b|_{1,\infty,T} + b_{\infty} \kappa)^2 \| v \|_{\mcT_h}^2
  \lesssim 
  (\tau_c h)^{-1}\| v \|_{\mcT_h}^2.
    \label{eq:bhtildevsbh-II}
\end{align}
Collecting the bounds~\eqref{eq:bhtildevsbh-I} and~\eqref{eq:bhtildevsbh-II}
proves the first inequality in~\eqref{eq:bh-b_est_upw}
while the second follows immediately from
$L^2$-norm extension property~\eqref{eq:l2norm-extension-property}.
\qed\end{proof}
With these preparations at hand, we can now show that  ---similar to the extended $L^2$-norm--- the \emph{extended streamline
diffusion norm} can be controlled by the surface streamline diffusion norm if suitable stabilization terms
are added:
\begin{lemma}
Let $\btilde_h$ be an element-wise constant vector field satisfying~\eqref{eq:b-bhtilde_assump}. 
For $v \in \PPdck$  we have the estimate
\begin{align}
    \frac{1}{h}\|\phi_b^{\onehalf} \widetilde{b}_h \cdot \nabla v\|_{\mcT_h}^2 
    \lesssim 
    &\|\phi_b^{\onehalf} b_h \cdot \nablash v\|_{\mcK_h}^2 
    + \frac{b_{\infty}}{h}\|[v]\|_{\mcF_h}^2
    + b_{\infty} h\|[n_F \cdot \nabla v]\|_{\mcF_h}^2
    \nonumber
    \\
    &\quad
    + b_{\infty} \|n_{\Gamma_h} \cdot \nabla v\|_{\mcT_h}^2 
    + \tn v \tnuph^2.
\end{align}
\label{lem:streamline_stab}
\end{lemma}
\begin{proof}
We start by applying Lemma \ref{lem:inverse_discontinuous}
to the discrete function $\phi_b^{\onehalf} \widetilde{b}_h \cdot \nabla v \in V_h$,
yielding
\begin{align}
    \frac{1}{h} \|\phi_b^{\onehalf} \widetilde{b}_h \cdot \nabla v\|_{\mcT_h}^2 
    &\lesssim  \|\phi_b^{\onehalf} \widetilde{b}_h \cdot \nabla v\|_{\mcK_h}^2 
    + \|[\phi_b^{\onehalf} \widetilde{b}_h \cdot \nabla v]\|_{\mcF_h}^2
    + h \|n_{\Gamma_h}\cdot\nabla(\phi_b^{\onehalf} \widetilde{b}_h \cdot \nabla v)\|_{\mcT_h}^2 
    \nonumber 
    \\
    &= I + II + III.
    \label{eq:sdest-I-II-III}
\end{align}
Recalling~\eqref{eq:bh-b_est_upw} and the inverse estimate~\eqref{eq:inversegamma}, we see that term $I$ can be bounded by
\begin{align}
   I 
   \lesssim  
   \|\phi_b^{\onehalf} b_h \cdot \nablash v\|_{\mcK_h}^2 + 
   \dfrac{1}{h}\|\phi_b^{\onehalf} (\widetilde{b}_h - b_h) \cdot \nabla v\|_{\mcT_h}^2 
   \lesssim
   \|\phi_b^{\onehalf} b_h \cdot \nablash v\|_{\mcK_h}^2 + (\tau_c h)^{-1}\| v\|_{\mcT_h}^2
   \label{eq:ghost_streamline_I}
\end{align}
Next, to estimate term $II$, 
we can switch between $\btilde_h$ and $b^e$ and control the difference
term again by the $\tn \cdot \tnuph$ norm. Using the fact that $[b^e] = 0$, we obtain
\begin{align}
    II 
    &\lesssim 
     b_{\infty} h \|[\nabla v]\|_{\mcF_h}^2
    + \|[\phi_b^{\onehalf} (b^e -\widetilde{b}_h) \cdot \nabla v]\|_{\mcF_h}^2
    = II_a + II_b.
\end{align}
Now let $P_F \nabla v$ be the part of $\nabla v$ that is tangential to $F$,
so that $\nabla v = P_F \nabla v + (n_F \cdot \nabla v) n_F$.
Applying an inverse estimate similar to~\eqref{eq:inverse} to $\|P_F \nabla v \|_{F}$ allows us to bound $II_a$ by
\begin{align}
    II_a 
    & = 
     h b_{\infty} \|[P_F \nabla v]\|_{\mcF_h}^2 +  h b_{\infty} \|[n_F \cdot \nabla v]\|_{\mcF_h}^2
    \lesssim 
    \frac{b_{\infty}}{h} \|[v]\|_{\mcF_h}^2 + h b_{\infty} \|[n_F \cdot \nabla v]\|_{\mcF_h}^2.
    \label{eq:ghost_streamline_II_a}
\end{align}
For the second contribution $II_b$, we first recall the
assumptions~\eqref{eq:bhassumptions} which together with
the inverse estimate~\eqref{eq:inverseandtrace} yields
\begin{align}
    II_b 
    \lesssim  
     \frac{\phi_b}{h}
    (b_\infty \kappa + |b|_{1,\infty,\Gamma})^2 \| v\|^2_{\mcT_h}
    \lesssim  
     \frac{\phi_b}{h} \tau_c^{-2}
    \lesssim
    (\tau_c h)^{-1}\|v\|_{\mcT_h}^2.
    \label{eq:ghost_streamline_II_b}
\end{align}
where in the last step we again used that $\tau_c^{-1} \phi_b \lesssim
1$ by our assumption~\ref{eq:mesh-resol-assump} on the mesh
resolution.

Finally, we turn to the remaining term $III$ in~\eqref{eq:sdest-I-II-III}.
Similar to $\btilde_h$, let $\widetilde{n}_{\Gamma_h}$ be an element-wise
constant approximation of $n_{\Gamma_h}$ such that 
$\|n_{\Gamma_h} -\widetilde{n}_{\Gamma_h}\|_{\infty} \lesssim \kappa h$.
Then
$
\widetilde{n}_{\Gamma_h}\cdot\nabla(\widetilde{b}_h \cdot \nabla v)
=
\widetilde{b}_h \cdot\nabla(\widetilde{n}_{\Gamma_h}\cdot \nabla v)
$
and thus in combination with multiple applications
of the inverse estimate~\eqref{eq:inverse} 
we derive that
\begin{align}
    III &= 
    h \|n_{\Gamma_h}\cdot\nabla(\phi_b^{\onehalf} \widetilde{b}_h \cdot \nabla v)\|_{\mcT_h}^2 
    \\
    &\lesssim h \phi_b\|\widetilde{n}_{\Gamma_h}\cdot\nabla(\widetilde{b}_h \cdot \nabla v)\|_{\mcT_h}^2
    + h \phi_b\|(n_{\Gamma_h} - \widetilde{n}_{\Gamma_h})\cdot\nabla(\widetilde{b}_h \cdot \nabla v)\|_{\mcT_h}^2
    \\
    &\lesssim 
    h \phi_b
    \|\widetilde{b}_h \cdot\nabla(\widetilde{n}_{\Gamma_h}\cdot \nabla v)\|_{\mcT_h}^2
    +\phi_b \kappa^2  h^3 \|\nabla(\widetilde{b}_h \cdot \nabla v)\|_{\mcT_h}^2
    \\
    &\lesssim 
    h^{-1} \phi_b b_{\infty}^2
    \|\widetilde{n}_{\Gamma_h}\cdot \nabla v\|_{\mcT_h}^2
    +  \phi_b \kappa^2 b_{\infty}^2 h\|\nabla v\|_{\mcT_h}^2
    \\
    &\lesssim 
    b_{\infty}
    \|\widetilde{n}_{\Gamma_h}\cdot \nabla v\|_{\mcT_h}^2
    + (\tau_c h)^{-1}  
    \| v\|_{\mcT_h}^2
    \label{eq:ghost_streamline_III-inter}
    \\
    &\lesssim 
    b_{\infty}
    \|n_{\Gamma_h}\cdot \nabla v\|_{\mcT_h}^2
    +
    b_{\infty}\|(n_{\Gamma_h}-\widetilde{n}_{\Gamma_h})\cdot \nabla v\|_{\mcT_h}^2
    + 
    (\tau_c h)^{-1} \|v\|_{\mcT_h}^2
    \\
    &\lesssim 
    b_{\infty}
    \|n_{\Gamma_h}\cdot \nabla v\|_{\mcT_h}^2
    + 
    (\tau_c h)^{-1}
    \|v\|_{\mcT_h}^2.
    \label{eq:ghost_streamline_III}
\end{align}
Here, we used the fact that thanks to assumption~\eqref{eq:mesh-resol-assump}, 
$
\phi_b \kappa^2 b_{\infty}^2 \lesssim \tau_c^{-1}
$
to pass to \eqref{eq:ghost_streamline_III-inter},
and 
in the last step, Lemma~\ref{lem:bh-b_est_upw} was employed.
Finally, collecting the obtained bounds
\eqref{eq:ghost_streamline_I},
\eqref{eq:ghost_streamline_II_a},
\eqref{eq:ghost_streamline_II_b}, and
\eqref{eq:ghost_streamline_III}
and applying Lemma~\ref{eq:bh-b_est_upw} once more
to bound $(\tau_c h)^{-1} \|v\|_{\mcT_h}^2$ by $\tn v \tnuph^2$
yields the desired estimate.
\qed\end{proof}
Motivated by Lemma \ref{lem:streamline_stab}, we define now our stabilization for the advection part and set
\begin{align}
    s_h^{a}(v, w) =
     \frac{\gamma^b_0 b_{\infty}}{h} ([v], [w])_{\mcF_h}
     + \gamma^b_1 h b_{\infty} ([n_F \cdot \nabla v], [n_F \cdot \nabla w])_{\mcF_h} 
     + \gamma^b_n b_{\infty} (n_{\Gamma_h} \cdot \nabla v, n_{\Gamma_h} \cdot \nabla w)_{\mcT_h}.
    \label{eq:ghost_ar3}
\end{align}
Thus, combining the reaction- and advection-related stabilization forms 
suggest considering
\begin{align}
    s_h(v, w) 
    &= (\gamma^c_0 \tau_c^{-1} + \dfrac{\gamma^b_0 b_{\infty}}{h})([v], [w])_{\mcF_h} 
    + \gamma^b_1  b_{\infty} h ([n_F \cdot \nabla v], [n_F \cdot \nabla w])_{\mcF_h}
    \nonumber
    \\
    &\quad
    + (\gamma^c_n \tau_c^{-1} h + \gamma^b_n b_{\infty}) (n_{\Gamma_h} \cdot \nabla v, n_{\Gamma_h} \cdot \nabla w)_{\mcT_h}
    \label{eq:sh-def-prelim}
\end{align}
as a candidate for the total stabilization form.
However, thanks to Assumption~\eqref{eq:mesh-resol-assump}, we only need to 
consider the advection part since
$\tau_c^{-1} \leqslant \phi_b^{-1} = \tfrac{b_{\infty}}{h}$,
leading us to the final definition of $s_h$.
\begin{definition}[Stabilization form $s_h$]
    \label{def:sh-def}
The stabilization $s_h$ is given by
\begin{align}
    s_h(v, w) 
    &= \dfrac{\gamma_0 b_{\infty}}{h}([v], [w])_{\mcF_h} 
    + \gamma_1  b_{\infty} h ([n_F \cdot \nabla v], [n_F \cdot \nabla w])_{\mcF_h}
    + \gamma_n b_{\infty} (n_{\Gamma_h} \cdot \nabla v, n_{\Gamma_h} \cdot \nabla w)_{\mcT_h}
    \label{eq:sh-def}
\end{align}
with $\gamma_0, \gamma_1$, and $\gamma_n$  being dimensionless, positive stability
parameters which depend on $k, d$, the quasi-uniformity of $\mcT_h$, and the curvature of $\Gamma$.
\end{definition}
For future reference, we summarize our discussion in the following corollary.
\begin{corollary}
Both the extended $L^2$ and streamline diffusion norm can be
controlled by augmenting the standard streamline diffusion
norm~\eqref{eq:sd-norm-def} with the semi-norm $|\cdot |_{s_h}$ induced by \eqref{eq:sh-def} in the sense 
that
\begin{align}
    \frac{1}{h}\|\phi_b^{\onehalf} \widetilde{b}_h \cdot \nabla v\|_{\mcT_h}^2 +
    \frac{1}{h}\|\tau_c^{-\onehalf}v\|_{\mcT_h}^2 \lesssim \tn v \tnsdh^2
\end{align}
holds for $v \in V_h$.
\label{cor:sh-stab-norm-est}
\end{corollary}

\subsection{$\mathbf{L^2}$-coercivity for $A_h$} 
\label{ssec:L2-coerc}
Next, we wish to show that the bilinear form $A_h$ is coercive with respect 
to the $\tn \cdot \tnuph$ norm. As usual, the approach is based
on exploiting the skew-symmetry of the advection-related terms via
an integration by parts argument,
but in contrast to the classical Euclidean case,
an additional term arises from the fact that the
co-normal vectors $n_E+$ and $n_E^-$ are not co-planar.
More precisely, for the co-normal jump, the following result holds.
\begin{lemma}
    Given a vector field $b$ and scalar fields $v$ and $w$
    for which we assume the edge jump and average to be well-defined.
     Then
\begin{equation}
    [vb ; n_E] = [v]\{b ; n_E\} + \{v\}[b ; n_E]
\end{equation}
and
\begin{align}
    [b v w ; n_E] = \{b ; n_E\}[v]\{w\} + \{b ; n_E\}\{v\}[w]  + [b ; n_E]\{vw\}.
    \label{eq:threejump}
\end{align}
\label{lem:threejump}
\end{lemma}
\begin{proof}
Writing out the definition of surface jumps and averages shows that
\begin{align}
    [v]\{b ; n_E\} + \{v\}[b ; n_E] &= (v^+ - v^-)\frac{1}{2}(n_E^+ b^+ - n_E^- b^-) + \frac{1}{2}(v^+ + v^-)(n_E^+ b^+ + n_E^- b^-)\\
    &= \frac{1}{2}(v^+ n_E^+ b^+ - v^+  n_E^- b^- - v^- n_E^+ b^+ + v^- n_E^- b^- \\
    &\hspace{5mm}+ v^+ n_E^+ b^+ + v^+ n_E^- b^- + v^- n_E^+ b^+ + v^- n_E^- b^-)\nonumber\\
    &= v^+ n_E^+ b^+ + v^- n_E^- b^- = [vb; n_E].
\end{align}
Inserting $vw$ for $v$ in the above argument and applying the standard equality for jumps and averages, $[vw] = \{v\}[w] + [v]\{w\}$, yields equation \eqref{eq:threejump}.
\qed\end{proof}

Using Lemma~\ref{lem:threejump} and integrating $(\bhnablash v, w)$ by parts, we see that 
\begin{align}
    (b_h \cdot \nablash v, w)_{\mcK_h} &= -(v, b_h \cdot \nablash w)_{\mcK_h} - (\nablash \cdot b_h v, w)_{\mcK_h} +
    \int_{\mcE_h} [b_h v w ; n_E] \mrm{d}\mcE_h\nonumber\\
    &= -(v, b_h \cdot \nablash w)_{\mcK_h} - (v, \nablash \cdot b_h w)_{\mcK_h} \label{eq:advection-ibp}\\
    &\quad + (\{b_h ; n_E\}[v],\{w\})_{\mcE_h} + (\{b_h ; n_E\}\{v\},[w])_{\mcE_h}\nonumber\\
    &\quad + (\{vw\}, [b_h ; n_E])_{\mcE_h}.\nonumber
\end{align}
If we insert~\eqref{eq:advection-ibp} into~\eqref{eq:ahI}, we see that $a_h(\cdot, \cdot)$ is equivalent to
\begin{align}
    a_h(v, w) &= (c_h v, w)_{\mcK_h}  + (b_h \cdot \nablash v , w)_{\mcK_h} 
    -(\{b_h ; n_E\} [v], \{w\})_{\mcE_h}
    + \frac{1}{2}(|\{b_h ; n_E\}|[v], [w])_{\mcE_h}
    \label{eq:ahIa}
    \\
    &= (v, (c_h - \nablash \cdot b_h)  w)_{\mcK_h}  - (v,  b_h \cdot \nablash w)_{\mcK_h}
    + (\{b_h ; n_E\}\{v\}, [w])_{\mcE_h}  
    \label{eq:ahII}
    \\
    &\quad 
    + \frac{1}{2}(|\{b_h ; n_E\}|[v], [w])_{\mcE_h}
    + (\{vw\}, [b_h ; n_E])_{\mcE_h}.\nonumber
\end{align}
Thus, by combining one half of both~\eqref{eq:ahIa} and \eqref{eq:ahII}, the bilinear form $a_h(\cdot, \cdot)$
can be divided into a symmetric and a skew-symmetric part,
\begin{align}
a_h(v, w) 
&=  
a_h^{\mrm{sy}}(v, w)  + a_h^{\mrm{sk}}(v, w),
\label{eq:ah-decomposition}
\intertext{where}
a_h^{\mrm{sy}}(v, w)  
&=  ((c_h - \frac{1}{2}\nablash \cdot b_h)v, w)_{\mcK_h} + \frac{1}{2}(|\{b_h ; n_E\}|[v], [w])_{\mcE_h}
+ \dfrac{1}{2}(\{vw\}, [b_h ; n_E])_{\mcE_h},
\label{eq:ahsy-def}
\\
 a_h^{\mrm{sk}}(v, w) &= \frac{1}{2}(b_h \cdot \nablash v, w)_{\mcK_h} - \frac{1}{2} (v, b_h \cdot \nablash w)_{\mcK_h}
\label{eq:ahsk-def}
 \\
 &\quad - \frac{1}{2}(\{b_h ; n_E\} [v], \{w\})_{\mcE_h} + \frac{1}{2}(\{b_h ; n_E\}\{v\}, [w])_{\mcE_h}.
 \nonumber
\end{align}
Note that even with the standard
assumption~\eqref{eq:bh-essinf-assump}, it is not obvious that the
symmetric part is positive definite due to the last term
in~\eqref{eq:ahsy-def} arising from the lack of co-planarity of the
edge normal vectors.
Nevertheless, the next lemma shows that thanks to the stabilization term $s_h$,  
the symmetric part $a_h^{\mrm{sy}}$ is in fact $L^2$ coercive.
\begin{lemma}
If the geometry assumption~\eqref{eq:discrete_surf_ass}
and assumption~\eqref{eq:bh-essinf-assump} on the coefficients
$c_h$ and $b_h$ hold, 
the stabilized bilinear form $A_h = a_h + s_h$ is coercive with respect to the stabilized upwind norm; that is,
\begin{align}
    A_h(v, v) \gtrsim c_0\tau_c\tn v\tnuph^2 \quad \foralls v \in V_h.
    \label{eq:upwindstability}
\end{align}
\label{lem:upwindstability}
\end{lemma}
\begin{proof}
    From decomposing $a_h$ into its
    symmetric and skew-symmetric part, cf.~\eqref{eq:ah-decomposition},
    it follows that
\begin{align}
A_h(v, v) 
&= ((c_h - \frac{1}{2}\nablash \cdot b_h)v, v)_{\mcK_h} +  \frac{1}{2}(|\{b_h ; n_E\}|[v], [v])_{\mcE_h} +
(\{v^2\}, [b_h ; n_E])_{\mcE_h} +
s_h(v, v)\\
&\geqslant c_{0,h} (v, v)_{\mcK_h} + \frac{1}{2}(|\{b_h ; n_E\}|[v], [v])_{\mcE_h} + s_h(v, v)
+ (\{v^2\}, [b_h ; n_E])_{\mcE_h}.
\\
& \gtrsim c_{0,h}\tau_c\tn v\tnuph^2
+ (\{v^2\}, [b_h ; n_E])_{\mcE_h}.
\label{eq:ah-L2-coerciv-step-I}
\end{align}
The remaining $(\{v^2\}, [b_h ; n_E])_{\mcE_h}$ term in~\eqref{eq:ah-L2-coerciv-step-I}
can be handled by 
combining the inverse estimates~\eqref{eq:inverseE} and~\eqref{eq:inverseandtrace} with
\eqref{eq:l2norm-extension-property} leading to
\begin{align}
    (\{v^2\}, [b_h ; n_E])_{\mcE_h}
    &\lesssim
    h^{-2} \|v\|_{\mcT_h}^2 C_b h^{k_g+1}
    =
    (h \tau_c)^{-1} \|v\|_{\mcT_h}^2  \tau_c C_b h^{k_g}
    \lesssim
    \tn v \tnuph  \tau_c c_{0,h} \dfrac{C_b h^{k_g}}{c_{0,h}}.
\end{align}
Inserting this into \eqref{eq:ah-L2-coerciv-step-I} we see that
\begin{align}
A_h(v, v)
&\gtrsim \Bigl(1 - \dfrac{C_b h^{k_g}}{c_{0,h}}\Bigr)c_{0,h}\tau_c\tn v\tnuph^2
\gtrsim c_{0,h}\tau_c\tn v\tnuph^2
\end{align}
whenever $h$ is small enough.
\qed\end{proof}

\subsection{Inf-sup condition for $A_h$}
With the newly constructed stabilization form $s_h$ at our disposal 
we are now in the position to derive the main stability 
result for the proposed CutDG formulation, namely that
discrete bilinear form $A_h$ satisfies a \emph{geometrically robust} inf-sup condition with respect
to the stabilized streamline diffusion norm~$\tn\cdot\tnsdh$.
\begin{theorem}
Let $A_h$ be the stabilized discrete bilinear form given in~\eqref{eq:Ah-def}
with $s_h$ defined by~\eqref{eq:sh-def}.
Then for $v \in V_h$, we have that
\begin{align}
    c_0 \tau_c \tn v \tnsdh \lesssim \sup_{w \in V_h} \frac{A_h(v, w)}{\tn w \tnsdh},
\end{align}
where it is implicitly understood that the supremum excludes the case $w = 0$.
\label{thm:infsup_ar}
\end{theorem}
\begin{proof}
The statement is clearly true if given $v \in V_h \setminus \{0\}$ we can construct a function $w \in V_h$ such that
\begin{align}
    c_0 \tau_c \tn v \tnsdh \tn w \tnsdh \lesssim A_h(v, w).
\end{align}
The construction will be performed in three steps.

\textbf{Step 1:}
First, we choose $w_1 = v$. Then, by Lemma~\ref{lem:upwindstability},
\begin{align}
    c_0 \tau_c \tn v \tnuph \tn w_1 \tnuph \lesssim A_h(v, w_1).
    \label{eq:inf-sup-proof-step-1}
\end{align}

\textbf{Step 2:} 
Next, we set $w_2 = \phi_b \widetilde{b}_h \cdot \nabla v$ which is a
permissible test function since $\widetilde{b}_h \in \PPdczero$.
Inserting $w_2$ into $A_h$ and adding $\bhnablash v - \bhnabla v = 0$
to the convection term
gives
\begin{align}
    A_h(v, w_2) 
    &= \| \phi_b^{\onehalf} b_h \cdot \nablash v\|_{\mcK_h}^2
    + \phi_b (b_h \cdot \nablash v, \bigl( (\btilde_h - b_h)\cdot \nabla v\bigr)_{\mcK_h} 
    \\
    &\quad +(cv, w_2)_{\mcK_h}
    -(\{b_h ; n_E\} [v], \{w_2\})_{\mcE_h}\nonumber
    + \frac{1}{2}(|\{b_h ; n_E\}|[v], [w_2])_{\mcE_h}
    + s_h(v, w_2) \nonumber
    \\
    &= 
    \| \phi_b^{\onehalf} b_h \cdot \nablash v\|_{\mcK_h}^2
    + I + II + III + IV + V.
    \label{eq:Ah_w2}
\end{align}
Regarding the term $I$,
a successive application of the Cauchy--Schwarz inequality, inverse estimate~\eqref{eq:inversegamma}, and Lemma~\ref{lem:bh-b_est_upw} 
gives us
\begin{align}
    I
    &\lesssim
    \| \phi_b^{\onehalf} b_h \cdot \nablash v\|_{\mcK_h}  
    \cdot
    \biggl(\dfrac{\phi_b}{h}\biggr)^{\onehalf}\| \btilde_h - b_h \|_{0,\infty, \mcT_h} \|\nabla v \|_{\mcT_h}
    \\
    &\lesssim
    \| \phi_b^{\onehalf} b_h \cdot \nablash v\|_{\mcK_h} \tn v \tnuph
    \lesssim
    \tn v \tnuph \tn v \tnsdh.
    \label{eq:est_Ia}
\end{align}
Turning to the remaining terms $II$--$V$ in~\eqref{eq:Ah_w2},
let us assume for the moment that $w_2$ satisfies the stability estimate
\begin{align}
    \tn w_2 \tnsdasth \lesssim \tn v \tnsdh.
    \label{eq:stab_est_w2}
\end{align}
Then it is easy to see that 
\begin{align}
    |II| + |III| + |IV| + |V|
    \lesssim
    \tn v \tnuph \tn w_1 \tnsdasth 
    \lesssim
    \tn v \tnuph \tn v \tnsdh.
\end{align}
With the derived estimates for $|I|$ to $|V|$ at our disposal,
we now see that the identity 
$\tn v \tnsdh^2 =
\| \phi_b^{\onehalf} b_h \cdot \nablash v\|_{\mcK_h}^2  
+ \tn v \tnuph^2$
together with a scaled Young inequality of the form $ab \leqslant
\delta a^2 + \tfrac{1}{4 \delta} b^2$ leads to 
\begin{align}
    A_h(v, w_2)
    &= 
    \| \phi_b^{\onehalf} b_h \cdot \nablash v\|_{\mcK_h}^2  
    + I + II + III + IV + V 
    \\
    &\gtrsim 
    \| \phi_b^{\onehalf} b_h \cdot \nablash v\|_{\mcK_h}^2  
    - C \delta \tn v \tnsdh^2 - \dfrac{C}{4\delta} \tn v \tnuph^2
    \\
    &=
    (1 - C \delta)\| \phi_b^{\onehalf} b_h \cdot \nablash v\|_{\mcK_h}^2  
    - ( C \delta  + \dfrac{C}{4\delta}) \tn v \tnuph^2.
    \\
    &=
    \dfrac{1}{2}\| \phi_b^{\onehalf} b_h \cdot \nablash v\|_{\mcK_h}^2  
    - \dfrac{1+C^2}{2} \tn v \tnuph^2,
    \label{eq:inf-sup-proof-step-2}
\end{align}
where in the last step, we picked $\delta = \tfrac{1}{2C}$ with $C$ being some positive constant.

\textbf{Step 3:}
Finally, a suitable $w_3$ can be constructed by setting
$w_3 = w_1 + \epsilon c_0 \tau_c w_2$. 
The stability estimate~\eqref{eq:stab_est_w2} ensures that 
$
\tn w_3\tnsdh \lesssim \tn v \tnsdh
+ \epsilon c_0 \tau_c \tn v \tnsdh \leqslant (1+\epsilon)\tn v \tnsdh
$
and thanks to~\eqref{eq:inf-sup-proof-step-1} and~\eqref{eq:inf-sup-proof-step-2},
we conclude that $w_3$ satisfies
\begin{align}
  A_h(v, w_3)
  &\geqslant
  (1- \epsilon \widetilde{C})c_0 \tau_c   \tn v \tnuph^2
  + \dfrac{\epsilon}{2}c_0 \tau_c \| \phi_b^{\onehalf} \bnabla v \|_{\Omega}^2
  \nonumber
  \\
  &\gtrsim c_0\tau_c \tn v \tnsdh^2
  \gtrsim c_0 \tau_c \tn v \tnsdh \tn w_3 \tnsdh
\end{align}
for some constant $\widetilde{C}$ and $\epsilon > 0$ small enough.
To complete the proof, we only need to show that the stability estimate~\eqref{eq:stab_est_w2} holds.

{\bf Estimate~\eqref{eq:stab_est_w2}.}
Unwinding the definition of
 $\tn \cdot \tnsdasth$, cf.~\eqref{eq:sdhast-def},
leaves us with the following~5 terms to estimate,
\begin{align}
    \tn w_2 \tnsdast^2 
    &= \tau_c^{-1} \|w_2\|_{\mcK_h}^2 
    + \|\phi_b^{\onehalf} b_h \cdot \nablash w_2\|_{\mcK_h}^2
    + b_{\infty} \|w_2\|_{\partial \mcK_h}^2 
    + \phi_b^{-1}\| w_2\|_{\mcK_h}^2
    + |w_2|_{s_h}^2
\\
&=
I + II + III + IV + V.
\end{align}
To apply Lemma~\ref{cor:sh-stab-norm-est},
we will now show that each of the contributions $I$--$V$ can be bounded
by $\tfrac{1}{h}\|\phi_b^{\onehalf} \btildenabla v \|_{\mcT_h}^2$.
We start with the first term, where the inverse estimate~\eqref{eq:inversegamma}
immediately implies that
\begin{align}
    I 
    &= 
    \underbrace{\tau_c^{-1} \phi_b}_{\leqslant 1} \|\phi_b^{\onehalf} \btildenabla v \|_{\mcK_h}^2
    \lesssim
    \dfrac{1}{h}\|\phi_b^{\onehalf} \btildenabla v \|_{\mcT_h}^2.
\end{align}
For the second term, a combination of the two inverse
estimates~\eqref{eq:inverseE} and~\eqref{eq:inverseandtrace} together with the
definition~\eqref{eq:def-phi_b} of $\phi_b$ and the stability estimate~\eqref{eq:bhassumptions} for
$\btilde_h$ leads to 
\begin{align}
    II 
    &= \phi_b \| \bhnabla (\phi_b \btildenabla v) \|_{\mcK_h}^2
    \lesssim
    \phi_b \dfrac{b_{\infty}^2}{h^2} 
    \dfrac{1}{h} \| \phi_b \btildenabla v \|_{\mcT_h}^2
    =
    \underbrace{\phi_b^2 \dfrac{b_{\infty}^2}{h^2}}_{=1} 
    \dfrac{1}{h} \| \phi_b^{\onehalf} \btildenabla v \|_{\mcT_h}^2.
\end{align}
Turning to the third term, a successive application of~\eqref{eq:inverseE} and~\eqref{eq:inverseandtrace}
allows us again to pass from $\mcE_h$ to $\mcT_h$, yielding 
\begin{align}
    III 
    &= 
    b_{\infty} \| \phi_b \btildenabla v \|_{\partial \mcK_h}^2
    \lesssim
    \Bigl(
        \underbrace{\dfrac{b_{\infty}}{h} \phi_b}_{= 1} \Bigr) 
    \dfrac{1}{h} \| \phi_b^{\onehalf} \btildenabla v \|_{\mcT_h}^2.
\end{align}
Next, $IV$ can be handled easily using~\eqref{eq:inversegamma},
\begin{align}
    IV 
    &= \phi_b^{-1} \| \phi_b \btildenabla v \|_{\mcK_h}^2
    \lesssim
    \dfrac{1}{h} \| \phi_b^{\onehalf} \btildenabla v \|_{\mcT_h}^2.
\end{align}
Recalling the definition 
of $s_h$ given in ~\eqref{eq:sh-def}, we see that
the remaining term $V$ is composed of three terms,
\begin{align}
    V
    &= \dfrac{\gamma_0 b_{\infty}}{h}\|[w_2]\|_{\mcF_h}^2
    + \gamma_1 h b_{\infty} \|[n_F \cdot \nabla w_2]\|_{\mcF_h}^2
    + \gamma_n b_{\infty}
    \|n_{\Gamma_h} \cdot \nabla w_2\|_{\mcT_h}^2
    \\
    &=
    V_a + V_b + V_c.
\end{align}
All three contributions to $V$ can be treated very similarly.
Using $\eqref{eq:inverseandtrace}$ once more, we
obtain for $V_a$ the bound
\begin{align}
    V_a 
    &= \dfrac{\gamma_0 b_{\infty}}{h}\|[\phi_b \btildenabla v]\|_{\mcF_h}^2
    \lesssim
    \underbrace{\dfrac{\gamma_0 b_{\infty}}{h} \phi_b
    }_{\lesssim 1}
    \dfrac{1}{h} \|\phi_b^{\onehalf} \btildenabla v\|_{\mcT_h}^2.
\end{align}
The term $V_b$ can be treated similarly, but involves an 
additional application of the inverse estimate~\eqref{eq:inverse} which leads to
\begin{align}
    V_b \lesssim \dfrac{1}{h} \|\phi_b^{\onehalf} \btildenabla v\|_{\mcT_h}^2.
\end{align}
Finally,
\begin{align}
    V_c 
    &= 
    \gamma_n b_{\infty}
    \|n_{\Gamma_h} \cdot \nabla ( \phi_b \btildenabla v)\|_{\mcT_h}^2
    \lesssim
    \underbrace{
    \gamma_n b_{\infty}
    \dfrac{\phi_b}{h}}_{\lesssim 1}
    \dfrac{1}{h} \|\phi_b^{\onehalf} \btildenabla v\|_{\mcT_h}^2.
\end{align}
Collecting all bounds together with
Lemma~\ref{cor:sh-stab-norm-est} implies that
\begin{align}
I + \cdots  + IV \lesssim
\dfrac{1}{h}\|\phi_b^{\onehalf} \btildenabla v \|_{\mcT_h}^2
\lesssim \tn v \tnsdh^2
\end{align}
which concludes the proof of the stability estimate~\eqref{eq:stab_est_w2}.
\qed\end{proof}

\section{A priori error estimate}
\label{sec:aprioriest}
\subsection{A Strang-type lemma}
\label{ssec:erroradectionreaction}
As typical in the theoretical analysis of surface PDE discretizations,
the derivation of a priori estimates for the proposed CutDG method
departs from a Strang-type lemma,
which shows that the total discretization error 
is composed of an approximation error, a consistency error, and a
geometric error contribution.
\begin{lemma}
    For $s \geqslant 1$ let $u\in H^s(\Gamma)$  be the solution to the
    advection-reaction problem~\eqref{eq:continuous_problem}. Then the
    solution $u_h \in V_h$ to the discrete problem~\eqref{eq:Ah-def}
    satisfies the error estimate 
\begin{align}
    c_0 \tau_c \tn u^e - u_h \tnsd 
    &\lesssim 
    \inf_{v \in V_h} 
    \bigl( 
        \tn u^e - v \tnsdast 
    + |v|_{s_h}
    \bigr)
    \label{eq:strang_ar}
    \\
    &+ \sup_{w \in V_h \setminus \{0\}} \frac{a_h(u^e, w) - a_h(u, w^l)}{\tn w \tnsdh}
    + \sup_{w \in V_h \setminus \{0\}} \frac{l_h(w) - l(w^l)}{\tn w \tnsdh}.\nonumber
\end{align}
\label{lem:strang_ar}
\end{lemma}

\begin{proof}
First, we split the discretization error $u^e-u_h$ into an approximation
error $e_\pi = v - u^e$ and a discrete error $e_h = v - u_h$ and
obtain
$
c_0 \tau_c \tn u^e - u_h \tnsd 
    \lesssim \tn e_\pi \tnsd + c_0 \tau_c \tn e_h \tnsdh,
$
recalling that $c_0 \tau_c \lesssim 1$ by the definition of $\tau_c$.
To estimate the error contribution $\tn e_h \tnsdh$ further,
we want to invoke the inf-sup condition established in Theorem~\ref{thm:infsup_ar}.
First, observe that 
\begin{align}
    A_h(e_h, w) 
    &= a_h(v, w) - l_h(w) + s_h(v, w) 
    \\
    &= a_h(v -u^e, w) + a_h(u^e, w) - l_h(w) + s_h(v, w) 
    \\
    &= a_h(v -u^e, w) 
    + a_h(u^e, w) 
    - \bigl(a(u, w^l) - l(w^l)\bigr) 
    - l_h(w) + s_h(v, w) 
    \\
    &\leqslant 
    \tn v -u^e \tnsdast \tn w \tnsdh
    + \bigl(a_h(u^e, w) - a(u, w^l)\bigr) 
    + \bigl(l(w^l) - l_h(w)\bigr) + |v|_{s_h} |w|_{s_h} 
    \label{eq:error-split}
\end{align}
where we successively employed~\eqref{eq:Ah-def}, \eqref{eq:bilinear_exact}, and finally, \eqref{eq:A_h_boundedness}.
Inserting~\eqref{eq:error-split}
into the inf-sup condition
\begin{align}
    c_0 \tau_c \tn e_h \tnsdh &\lesssim \sup_{w \in V_h \setminus \{0\}} \frac{A_h(e_h, w)}{\tn w \tnsdh}.
    \label{eq:apriori_1}
\end{align}
yields the desired estimate.
\qed\end{proof}
In the remaining subsections, we will establish concrete estimates
for the approximation, consistency, and geometric error contributions.

\subsection{Approximation and consistency error estimates}
Next, we bound the approximation error for the quasi-interpolation
operator $\pi_h: L^2(\Gamma)  \to V_h$ constructed in Section~\ref{ssec:quasi-interpolation}.
\begin{lemma}
    Let $V_h = \PPdck$ and assume that $v \in H^s(\Gamma)$ with $s \geqslant 2$.
    Set $r = \min\{s, k+1\}$.
Then $\pi_h v$ satisfies the error estimate
\begin{align}
\tn v^e - \pi_h v \tnsdast \lesssim b_{\infty} h^{r - 1/2} \| v \|_{r, \Gamma}.
\label{eq:proj_estimates}
\end{align}
\label{lem:proj_estimates}
\end{lemma}
\begin{proof}
Setting $e_{\pi} = v - \pi_h v$ and unwinding the definition of $\tn \cdot
\tnsdast$ given in \eqref{eq:sdhast-def} leaves us with 5 terms to estimate,
\begin{align}
    \tn e_\pi \tnsdast^2 
    &=
    \phi_b^{-1}\| e_{\pi} \|_{\mcK_h}^2
    + \tau_c^{-1} \|e_{\pi}\|_{\mcK_h}^2 
    + \|\phi_b^{\onehalf} b_h \cdot \nablash  e_{\pi} \|_{\mcK_h}^2
    + b_{\infty} \|e_{\pi}\|_{\partial \mcK_h}^2 
    \\
    &=
    I + II + III + IV.
\end{align}
The first term can be simply estimated by combining the trace
inequality~\eqref{eq:traceF} with standard interpolation
estimate~\eqref{eq:L2T}, followed by a final application of the
co-area formula~\eqref{eq:extensionstability} with $\delta \sim h$,
leading to
\begin{align}
    I 
    &= b_{\infty} h^{-1} \| e_{\pi} \|_{\mcK_h}^2
    \lesssim  b_{\infty} h^{-1} 
    \bigl( h^{-1} \| e_{\pi} \|_{\mcT_h}^2 + h \| \nabla e_{\pi} \|_{\mcT_h}^2  \bigr)
    \\
    &\lesssim  b_{\infty} h^{2r-2} \| v^e\|_{r, \mcT_h}^2
    \lesssim  b_{\infty} h^{2r-2} \| v^e\|_{r, U_{\delta_h}(\Gamma)}^2
    \lesssim  b_{\infty} h^{2r-1} \| v\|_{r, \Gamma}^2.
\end{align}
The second and third terms  can be estimated in a similar fashion,
\begin{align}
    II 
    &= \underbrace{(\tau_c^{-1} \phi_b)}_{\leqslant 1} \phi_b^{-1} \| e_{\pi} \|_{\mcK_h}^2 
    \lesssim I 
    \lesssim  b_{\infty} h^{2r-1} \| v\|_{r, \Gamma}^2,
    \\
    III 
    &\lesssim 
     b_{\infty} h \| \nabla e_{\pi} \|_{\mcK_h}^2
    \lesssim
     b_{\infty} h \bigl( h^{-1} \| \nabla e_{\pi} \|_{\mcT_h}^2 + h \| \nabla \otimes \nabla e_{\pi} \|_{\mcT_h}^2  \bigr)
    \\
    &\lesssim  b_{\infty} h^{2r-2} \| v^e\|_{r, \mcT_h}^2
    \lesssim  b_{\infty} h^{2r-1} \| v\|_{r, \Gamma}^2.
\end{align}
Using the interpolation estimate~\eqref{eq:L2F} instead of~\eqref{eq:L2T}, also
the remaining term $IV$ can be treated similarly, 
\begin{align}
    IV 
    \lesssim  
    b_{\infty} \bigl( h^{-1} \| e_{\pi} \|_{\mcF_h}^2 + h \| \nabla e_{\pi} \|_{\mcF_h}^2  \bigr)
    \lesssim b_{\infty}  h^{2r-2} \| e_{\pi} \|_{r, \mcT_h}^2 
    \lesssim  b_{\infty} h^{2r-1} \| v\|_{r, \Gamma}^2.
\end{align}
\qed\end{proof}

\begin{lemma}
\label{lem:consistency_error}
Under the same assumptions as in Lemma~\ref{lem:proj_estimates},
the consistency error $|\pi_h v |_{s_h} $ can be bounded by
\begin{align}
|\pi_h v |_{s_h} 
\lesssim b_{\infty} ( 
     h^{r - \onehalf}+ h^{k_g + \onehalf}
     )\| v \|_{r, \Gamma}.
\label{eq:consistency_error}
\end{align}
\end{lemma}
\begin{proof}
Since $v \in H^s(\Gamma)$ with $s \geqslant 2$, 
both 
the expressions
$\|[\pi_h v]\|_{\mcF_h}$,
$\|[n_F \cdot \nabla \pi_h v]\|_{\mcF_h}$ and
$\|[n_{\Gamma}^e \cdot \nabla \pi_h v]\|_{\mcT_h}$ vanish
for the normal extension $v^e \in H^s(U_{\delta_h}(\Gamma))$.
Consequently,
\begin{align}
	| \pi_h v |_{s_h}^2
    \nonumber
    &= \dfrac{\gamma_0 b_{\infty}}{h}\|[\pi_h v - v^e]\|_{\mcF_h}^2
    + \gamma_1 h b_{\infty} \|[n_F \cdot \nabla (\pi_h v - v^e) \|_{\mcF_h}^2
    \\
    &\quad
    + \gamma_n b_{\infty}
    \| n_{\Gamma_h} \cdot \nabla \pi_h v - n_{\Gamma}^e \cdot \nabla v^e\|_{\mcT_h}^2
    =  I + II + III.
\end{align}
Successively applying interpolation estimate~\eqref{eq:L2F} with $k = 0$ and
stability estimate~\eqref{eq:extensionstability} with $\delta \sim h$, we see
that
\begin{align}
    I 
    \lesssim  \dfrac{\gamma_0 b_{\infty}}{h} h^{2r-1}\|v^e\|_{r, \mcT_h}^2
    \lesssim  \dfrac{\gamma_0 b_{\infty}}{h} h^{2r}\|v\|_{r, \Gamma}^2,
\end{align}
and similarly,
\begin{align}
    II 
    \lesssim 
    \gamma_1 h b_{\infty} h^{2r-3}\|v^e\|_{r, \mcT_h}^2
    \lesssim 
    \gamma_1 b_{\infty} h^{2r-1}\|v^e\|_{r, \Gamma}^2.
\end{align}
To estimate the remaining term $III$, we also
need to take into account the geometrical approximation assumption~\eqref{eq:discrete_surf_ass},
yielding
\begin{align}
    (\gamma_n b_{\infty})^{-1}
    III  
    &\lesssim 
    \| (n_{\Gamma_h} - n_{\Gamma}^e) \cdot \nabla \pi_h v\|_{\mcT_h}^2
    + \| n_{\Gamma}^e \cdot (\nabla v^e - \pi_h v)\|_{\mcT_h}^2
    \\
    &\lesssim
    h^{2 k_g}\|\nabla \pi_h v\|_{\mcT_h}^2
    + h^{2r - 2 }
     \| v^e \|_{r, \mcT_h}^2
    \\
    &\lesssim
    h^{2 k_g + 1}\|v\|_{1, \Gamma}^2
    +h^{2r - 1 }
    \| v \|_{r, \Gamma}^2.
\end{align}
\qed\end{proof}

\subsection{Geometric error estimates}
\label{ssec:geom-err-est}
Finally, we estimate the remaining geometric error contributions
originating from our geometry approximation assumptions~\ref{eq:discrete_surf_ass}.
\begin{lemma}
\label{lem:geometric_error_arbilinear}
    For $u \in H^1(\Gamma)$ and $w \in V_h$ we have that
\begin{align}
    |a_h(u^e, w) - a(u, w^l)| &\lesssim \tau_c^{\onehalf} h^{k_g+1} \|u\|_{1, \Gamma} \tn w \tnsd,
    \label{eq:geometric-err-ah}
    \\
    |l(w^l) - l_h(w)| &\lesssim \tau_c^{\onehalf} h^{k_g+1} \|f\|_\Gamma \tn w \tnsd.
    \label{eq:geometric-err-lh}
\end{align}
\end{lemma}
\begin{proof}
Recalling definition~\eqref{eq:ahI} of the discrete bilinear form $a_h$,
 find that
\begin{align}
    a_h(u, w) - a(u, w) 
    &= 
    \bigl(
        (c_h u^e, w)_{\Gamma_h} - (cu, w^l)_\Gamma
    \bigr)
    + \bigl(
     (b_h \cdot \nablash u^e, w)_{\Gamma_h} - (b \cdot \nablas u, w^l)_\Gamma
     \bigr)
    \nonumber
    \\
    &\quad -(\{b_h ; n_E\} [u^e], \{w\})_{\mcE_h}
    + \frac{1}{2}(|\{b_h ; n_E\}|[u^e], [w])_{\mcE_h}
    \nonumber\\
    &
    = I + II + III + IV.
\end{align}
Since $u^e \in H^{1}(U_{\delta}(\Gamma))$, 
we note that $[u^e] = 0$ and thus the contributions from $III$ and $IV$ vanish.
Turning to the first term $I$ and changing the integration domain from $\Gamma$ to $\Gamma_h$,
we can use assumptions
\eqref{eq:ch_assumption} and \eqref{eq:normeql2} to obtain the bound
\begin{align}
    I &= 
     (c_h u^e, w)_{\Gamma_h} - (|B| c u^e , w)_{\Gamma_h}
      = ((c_h - |B| c^e) u^e, w)_{\Gamma_h}
      \\
      &\lesssim \|c_h - |B| c^e\|_{L^\infty(\Gamma_h)} 
      \|u^e\|_{\Gamma_h} \tau_c^{\onehalf} \tau_c^{-\onehalf}\|w\|_{\Gamma_h}
      \lesssim h^{k_g+1} \|u\|_{\Gamma} \tau_c^{\onehalf} \tn w \tnsdh
\end{align}
Similarly for $II$, assumptions \eqref{eq:bh_assumption} and~\eqref{eq:normeqh1} can be employed to conclude that
\begin{align}
    II 
    &= (b_h \cdot \nablas u^e), w)_{\Gamma_h} - (|B| b^e \cdot B^{-T} \nablash u^e, w)_{\Gamma_h}
    \\
    &= ((b_h - |B|B^{-1} b^e) \cdot \nablas u^e, w)_{\Gamma_h}
    \lesssim h^{k_g+1} \|u\|_{1,\Gamma} \tau_c^{\onehalf} \tn v \tnsd.
\end{align}
Finally, estimate \eqref{eq:geometric-err-lh} can be obtained in the exact same way as the bound for $I$.
\qed\end{proof}

\subsection{A priori error estimate}
Combining the above bounds for the approximation, consistency and geometric errors
with the abstract Strang-type Lemma~\ref{lem:strang_ar},
we arrive at the final a priori error estimate.
\begin{theorem}
For $s \geqslant 2$,
let $u \in H^s(\Gamma)$ be the solution to \eqref{eq:bilinear_exact}, and let $u_h \in V_h = \PPdck$ be the discrete solution 
to \eqref{eq:Ah-def}.  With $r = \min\{s, k+1\}$,
the following a priori error estimate holds,
\begin{equation}
    c_0\tau_c^{-1}\tn u - u_h \tnsd 
    \lesssim 
    (
    b_{\infty}h^{r-\onehalf}
    +b_{\infty}h^{k_g+\onehalf}
    + \tau_c^{\onehalf} h^{k_g+\onehalf}
    )
    \|u\|_{r, \Gamma} 
    + \tau_c^{\onehalf} h^{k_g+1} \|f\|_\Gamma,
\end{equation}
\label{thm:apriori_ar}
\end{theorem}

\subsection{Construction of alternative ghost penalties}
\label{ssec:alt-ghost-penalties}
To make the theoretical analysis more concrete,
we have focused on the design of one particular ghost penalty by
starting from the volume normal gradient stabilization~\eqref{eq:ineqcontinuous}
originally proposed in~\cite{GrandeLehrenfeldReusken2018,BuermanHansboLarsonEtAl2018}.
Nevertheless, similar to the abstract framework developed
in~\cite{BuermanHansboLarsonEtAl2018,GrandeLehrenfeldReusken2018}, the
presented approach can easily be generalized to cover the design of
alternative ghost penalties.
More precisely, a close inspection of the proofs of 
Lemma~\ref{lem:upwindstability} ($L^2$ coercivity), Theorem~\ref{thm:infsup_ar} (inf-sup condition),
Theorem~\ref{thm:apriori_ar} (a priori error estimate), and 
Theorem~\ref{thm:condition-number} (condition number estimate)
reveals that our theoretical analysis
holds for any ghost penalty $s_h$ which satisfies the following three abstract assumptions:
\begin{itemize}
    \item \textbf{A1)} The ghost penalty $s_h$ extends the $L^2$ in the sense that
    \eqref{eq:l2norm-extension-property} holds,
    \begin{align}
        h^{-1}\|v\|_{\mcT_h}^2 &\lesssim \|v\|_{\mcK_h}^2 + \|[v]\|_{\mcF_h}^2 + |v|_{s_h}^2
    \label{eq:l2norm-extension-property-abstr} 
    \end{align}
    \item \textbf{A2)} The ghost penalty $s_h$ extends the streamline diffusion norm in the sense that
    \eqref{lem:streamline_stab} holds,
    \begin{align}
    \frac{1}{h}\|\phi_b^{\onehalf} \widetilde{b}_h \cdot \nabla v\|_{\mcT_h}^2 
    \lesssim 
    &\|\phi_b^{\onehalf} b_h \cdot \nablash v\|_{\mcK_h}^2 
    + \tn v \tnuph^2
    + |v|_{s_h}^2
    \end{align}
    \item  \textbf{A3)} The ghost penalty $s_h$ is weakly consistent in the sense that
    \eqref{eq:consistency_error} holds,
    \begin{align}
    |\pi_h v |_{s_h} 
    \lesssim b_{\infty} ( 
        h^{r - \onehalf}+ h^{k_g + \onehalf}
        )\| v \|_{r, \Gamma}
        \end{align}
\end{itemize}
For example, a suitable alternative ghost penalty can be constructed starting from
stabilization~$s_h$ introduced in~\cite{LarsonZahedi2019}, which combines
a facet-based ghost penalty term $s_{h,F}$ with a higher-order normal derivative
stabilization $s_{h, \Gamma}$ which is evaluated only on the discrete surface $\Gamma_h$,
\begin{align}
    s_h(v,w) &= s_{h,F}(v,w) + s_{h,\Gamma}(v,w),  
    \\
    s_{h,F}  &=  \sum_{j=1} c_{F,j} h^{2(j-1+\gamma)} ([\partial_n^j v], [\partial_n^j w])_{\mcF_h},
    \label{eq:shF-def}
    \\
    s_{h,\Gamma}  &=  \sum_{j=1} c_{\Gamma,j} h^{2(j-1+\gamma)} (\partial_n^j v, \partial_n^j w)_{\Gamma_h}.
\end{align}
Choosing $\gamma =1$ and extending the summation index $j$ in~\eqref{eq:shF-def} to $0$,
we can proceed as before and combine the equivalent of Lemma~\ref{lem:ineqcontinuous} from~\cite{LarsonZahedi2019} 
with the Oswald interpolant and several standard inverse estimates to establish \textbf{A1}.
As before, with a stabilized $L^2$ estimate in place, we can design suitable ghost penalty candidates for the extension
of the streamline diffusion norm by simply replacing $v_h$ with $\btilde_h\cdot \nabla$ into~\eqref{eq:l2norm-extension-property-abstr}
and switching between $\btilde_h$ and $b_h$ via an equivalent of Lemma~\ref{lem:bh-b_est_upw} to obtain
\begin{align}
    \dfrac{1}{h}\|\phi^{\onehalf} \btilde_h \cdot \nabla v_h\|_{\mcT_h}^2
    \lesssim
    &\|\phi_b^{\onehalf} b_h \cdot \nablash v\|_{\mcK_h}^2 
    + \dfrac{b_{\infty}}{h} |v_h|_{s_h}^2 + \tn v_h \tnuph^2.
\end{align} 
It is then easy to show that
$\widetilde{s}_h(v,w) = b_{\infty}h^{-1} s_h(v, w)$ 
satisfies \textbf{A1)}--\textbf{A3)}.

\begin{remark}
    \label{rem:agglomeration}
The previous alternative ghost penalty also opens up for the possibly
use of agglomeration techniques from~\cite{HeimannEngwerIppischEtAl2013,JohanssonLarson2013}.
Indeed, ghost penalty $s_{h, F}$ could be omitted if elements with
small surface intersection would be merged with elements having a
large surface intersection. Nevertheless, one would need to keep $s_{h, \Gamma}$
to gain control over the variation of the discrete function in surface normal direction.
\end{remark}

\section{Condition number estimate}
\label{sec::condition-number-est}
In the final part of our theoretical analysis, we will investigate the
scaling behavior and geometrical robustness of the condition number of
the system matrix associated with the proposed CutDG method.  More
precisely, we will show that the condition number can be bounded by $C
h^{-1}$ with a constant that is independent of how the surface cuts
the background mesh.  As in our previous
contribution~\cite{GuerkanStickoMassing2020}, our presentation is
inspired by the general approach described in~\cite{ErnGuermond2006}.

To define the system matrix $\mcA$ associated with $A_h$, we first
introduce standard piecewise polynomial basis $\{\phi_i\}_{i=1}^N$
associated with $V_h = \PPdck$ allowing us to write any $v \in V_h$ as $v
= \sum_{i=1}^N V_i \phi_i$ with coefficients $V = \{V_i\}_{i=1}^N \in
\RR^N$.  
Then $\mcA$ is defined by the relation
\begin{align}
  ( \mcA V, W )_{\RR^N}  = A_h(v, w) \quad \foralls v, w \in
  V_h.
  \label{eq:stiffness-matrix}
\end{align}
Thanks to the $L^2$ coercivity of $A_h$ proved in~Section~\ref{ssec:L2-coerc},
the matrix $\mcA$ induces a bijective linear mapping 
$\mcA:\RR^N \to \RR^N$ with its operator norm and condition number given by
\begin{align}
  \| \mcA \|_{\RR^N}
  = \sup_{V \in \RR^N}
  \dfrac{\| \mcA V \|_{\RR^N}}{\|V\|_{\RR^N}}
\quad \text{and}
\quad
  \kappa(\mcA) = \| \mcA \|_{\RR^N} \| \mcA^{-1} \|_{\RR^N},
  \label{eq:operator-norm-and-condition-number-def}
\end{align}
where it is again implicitly understood that the supremum excludes the case $V = 0$.

As a first ingredient, 
we need to recall the well-known estimate
\begin{align}
  h^{d/2} \| V \|_{\RR^N}
  \lesssim
   \| v \|_{L^2(\mcT_h)}
  \lesssim
   h^{d/2} \| V \|_{\RR^N},
  \label{eq:mass-matrix-scaling}
\end{align}
which holds for any quasi-uniform mesh $\mcT_h$ and $v\in V_h$.
The inequalities stated in~\eqref{eq:mass-matrix-scaling} enable us
to pass between the continuous $L^2$ norm of a finite element
functions $v_h$ and the discrete $l^2$ norm of its associated 
coefficient vectors $V$, which will be essential in proving
the following theorem.
\begin{theorem}
  The condition number of the system matrix $\mcA$ associated with~(\ref{eq:Ah-def})
  satisfies
  \begin{align}
   \kappa(\mcA)  \lesssim b_{\infty} (c_0 h)^{-1}
  \end{align}
  where the hidden constant is
  independent of the particular cut configuration.
    \label{thm:condition-number}
\end{theorem}
\begin{proof}
  We need to bound $\|\mcA\|_{\RR^N}$ and $\|\mcA^{-1}\|_{\RR^N}$.
  
  {\bf Estimate of $\|\mcA\|_{\RR^N}$.}  
  As a first step, we bound $A_h(v,w)= a_h(v,w) + s_h(v,w)$
  in terms of the rescaled $L^2$ norm $h^{-\onehalf} \| \cdot \|_{\mcT_h}$.
  Recalling definition~\eqref{eq:ahI},
  \begin{align}
    a_h(v, w) &= 
    (c_h v b_h \cdot \nablash v , w)_{\mcK_h} 
    -(\{b_h ; n_E\} [v], \{w\})_{\mcE_h}
    + \frac{1}{2}(|\{b_h ; n_E\}|[v], [w])_{\mcE_h} 
    \\
    &= I + II + III,
  \end{align}
  we see that thanks to the inverse estimates~\eqref{eq:inversegamma},
  the first term can be treated as follows:
  \begin{align}
    I 
    &\lesssim
      (\|c\|_{0,\infty,\Gamma} h^{-1} + 
      \|b\|_{0,\infty,\Gamma} h^{-2}
      )
      \|v\|_{\mcT_h} \|w\|_{\mcT_h}
    \\
    &\lesssim
      (\tau_c^{-1} h  + 
      \|b_h\|_{0,\infty,\Gamma_h}
      ) h^{-2}
      \|v\|_{\mcT_h} \|w\|_{\mcT_h}
      \lesssim
      b_{\infty}h^{-2}
      \|v\|_{\mcT_h} \|w\|_{\mcT_h}.
  \end{align}
  Here, assumptions~\eqref{eq:ch_assumption_prel} and~\eqref{eq:bh_assumption_prel} allowed us
  to switch from the discrete to the continuous coefficients in the first step, and in the second
  step, ~\eqref{eq:mesh-resol-assump} was used.
  Next, a successive application of ~\eqref{eq:inverseE} and~\eqref{eq:inverseandtrace}
  leads to
  \begin{align}
  II + III
  \lesssim 
  b_{\infty} h^{-1 } \|v\|_{\partial \mcT_h} \|w\|_{\partial \mcT_h} 
  \lesssim b_{\infty}h^{-2} \|v\|_{\mcT_h} \|w\|_{\mcT_h}.
  \end{align}
  Turning to $s_h(v, w)$, we simply observe that the bound
  \begin{align}
      s_h(v,w)  \lesssim b_{\infty}h^{-2} \|v\|_{\mcT_h} \|w\|_{\mcT_h}.
  \end{align}
  follows immediately from the definition of $s_h$, cf.~\eqref{eq:sh-def},
  and the inverse estimates~\eqref{eq:inverse},~\eqref{eq:inverseandtrace}.
  Collecting all estimates and applying~(\ref{eq:mass-matrix-scaling}), we have
\begin{align}
  A_h(v, w)
    \lesssim  b_{\infty} h^{-2}  \|v\|_{\mcT_h} \|w\|_{\mcT_h}
    \lesssim b_{\infty} h^{d-2}  \|V\|_{\RR^N} \|W\|_{\RR^N},
\end{align}
and therefore we can bound $\| \mcA \|_{\RR^N}$ by
\begin{align}
  \| \mcA \|_{\RR^N}
  &=
    \sup_{V \in \RR^N}
    \sup_{W \in \RR^N}
    \dfrac{(\mcA V, W)_{\RR^N}}{\|V\|_{\RR^N}\|W\|_{\RR^N}}
  = 
  \sup_{V \in \RR^N}
    \sup_{W \in \RR^N}
    \dfrac{A_h(v,w)}{\|V\|_{\RR^N}\|W\|_{\RR^N}}
    \lesssim
      b_{\infty} h^{d-2}.
      \label{eq:norm-A-matrix-est}
\end{align}

{\bf Estimate of $\|\mcA^{-1}\|_{\RR^N}$.}
The discrete coercivity result~(\ref{eq:upwindstability}) combined
with the $L^2$-extension property~(\ref{eq:l2norm-extension-property}) of $s_h$
implies that
\begin{align}
\label{eq:est-Ainv}
  A_h(v, v) \gtrsim
  c_0 \tau_c \tn v \tnuph
  \gtrsim c_0 h^{-1} \| v \|_{\mcT_h}^2
  \gtrsim c_0 h^{d-1} \| V \|_{\RR^N}^2,
\end{align}
and consequently,
  \begin{align}
  \|\mcA V  \|_{\RR^N}
  &= \sup_{W \in \RR^N}
    \dfrac{(\mcA V, W)_{\RR^N}}{\|W\|_{\RR^N}}
    \geqslant
    \dfrac{(AV, V)_{\RR^N}}{\|V\|_{\RR^N}}
    =
    \dfrac{A_h(v,v)}{\|V\|_{\RR^N}}
    \gtrsim 
    c_0 h^{d-1} \|V\|_{\RR^N},
  \end{align}
  which implies that $\| \mcA^{-1}\|_{\RR^N} \lesssim c_0^{-1} h^{1-d}$.
  Combined with~\eqref{eq:norm-A-matrix-est}, we arrive
  at the desired bound
  \begin{align}
    \| \mcA \|_{\RR^N} \| \mcA^{-1} \|_{\RR^N} \lesssim b_{\infty} (c_0h)^{-1}.
  \end{align}
\qed\end{proof}

\section{Numerical results}
\label{sec:numerical-results}
In this final section, we conduct several numerical experiments to
corroborate our theoretical findings.  First, we perform a series of
tests to assess the order of convergence of the proposed CutDG method.
Afterward, the scaling behavior of the condition number
for 3 different $k$ orders is investigated numerically. 
Finally, we examine the geometrical robustness of our method by
studying the sensitivity of the computed errors and the condition
number with respect to the cut configurations.  
Unless stated otherwise, the following values of
the stabilization parameters have been used:
\begin{equation}
    \gamma_0^{b} = \gamma_0^c = \gamma_0 = 5k^2, \quad
    \gamma_n^{b} = \gamma_n^c = \gamma_n = 1,    \quad
    \gamma_1     = \frac{1}{2}.
    \label{eq:penalty-parameters}    
\end{equation}
The open source finite element library deal.II~\cite{dealII94} was used to implement the CutDG method
and conduct all numerical experiments. 

\subsection{Convergence tests}
In the first series of experiments, we examine the experimental
order of convergence (EOC) for orders $k=1,2,3$ over two different geometries
employing the method of manufactured solutions.
For the first geometry $\Gamma$ we choose the unit sphere defined by the $0$ level set
of the scalar function
\begin{align}
    \phi = \sqrt{x^2 + y^2 + z^2} - R, \qquad R = 1.
    \label{eq:level-set-sphere}
\end{align}
The unit sphere is embedded into a cubic domain $\Omega = [-L, L]^3$
with $L=1.21$ with is tessellated by a structured Cartesian mesh $\widetilde{\mcT}_0$.
with an initial subdivision of 12 elements in each coordinate direction.
The second domain consists of a torus
described by the $0$ level set of the scalar function
\begin{align}
    \phi = \sqrt{z^2 + (\sqrt{x^2 + y^2} -R)^2} - r,
    \label{eq:level-set-torus}
\end{align}
with $R=1$ and $r=1/3$.
The surface is immersed into the domain $\Omega = [-W, W] \times [-W, W]\times [-H, H]$,
where $H= \alpha r$, $W = \alpha (R + r)$ and $\alpha = 1.03$.
The initial Cartesian mesh $\widetilde{\mcT}_0$ for $\Omega$ consist
of a subdivision of $[N_x^0, N_y^0, N_z^0] = [12, 12, 3]$ elements in each coordinate direction.

To manufacture a problem that works for both surface geometries,
we set the analytical solution $u$, the advection field $b$, and the reaction coefficient $c$ to
\begin{subequations}
\begin{align}
        u & = \frac{x y}{\pi} \tan^{-1}\left(\frac{z}{\sqrt{\epsilon}}\right), \label{eq:experiments-analytical-solution} \\
        b & = (-y, x, 0) \sqrt{x^2 + y^2}, \label{eq:experiments-advection}                                               \\
        c & = 1. \label{eq:experiments-reaction}
\end{align}
    \label{eq:experiments-data}
\end{subequations}
and computing the right-hand side $f$ according to~\eqref{eq:continuous_problem}.
Note that $\tan^{-1}\left(\frac{z}{\sqrt{\epsilon}}\right)$ varies from $-\pi/2$ to $\pi/2$ over a distance~$\sim \sqrt{\epsilon}$ at the equator: ${\{(x,y,z) \in \Gamma : z=0\}}$.
Thus, the parameter $\epsilon$ allows us to modulate the smoothness of the solution
along the equator and that the solution is discontinuous from a numerical point of view 
until $\sqrt{\epsilon} \sim h$, i.e. until the internal layer is resolved by the mesh.

Now, starting from the initial Cartesian mesh $\widetilde{\mcT}_0$ for each surface,
we generate a series of meshes $\widetilde{\mcT_l}$ with mesh size $h_l$ for $l = 0, 1, \ldots$ by setting
the number of subdivision $[N_x^l, N_y^l, N_z^l]$ in each dimension
to $[N_x^l, N_y^l, N_z^l] = \lfloor 2^{l/2}\rfloor \cdot [N_x^0, N_y^0, N_z^0]$.
On each generated mesh $\widetilde{\mcT}_l$, we 
extract the active mesh $\mcT_l$ and
compute for each order $k$ the numerical solution $u_l^k \in \PP_{\mrm{dc}}^k(\mcT_l)$
and the resulting experimental order of convergence
(EOC) defined by
\begin{align}
  \text{EOC}(l,k) =
  \dfrac{\log(E_{l-1}^k/E_{l}^k)}{\log(h_{l-1}/h_l)}, 
\end{align} 
where $E_l^k = \| e_l^k \| = \| u - u_l^k \|$ denotes the error of the numerical
approximation $u_k^p$ measured in a certain (semi-)norm $\|\cdot \|$.    
The error norms considered in our tests are the $L^2$
norm $\| \cdot \|_{\Gamma}$
and the streamline diffusion norm $\tn\cdot\tn_{\mrm{sd}}$.

For a smooth solution, $u$, corresponding to $\epsilon = 1$ the
experimental orders of convergence for order $k=1,2,3$ are recorded in
Figure~\ref{fig:convergence_plots_eps_large} and confirm the
theoretically predicted convergence rate $k+1/2$ derived in
Section~\ref{sec:aprioriest}.  The $L^2$ convergence rate is
even half an order higher than expected and optimal, but we point out that
this is an often observed phenomenon on structured meshes that cannot
be expected on more general meshes, see~\cite{Peterson1991}.  A
visualization of the discrete solution on the surface and (part of)
the background mesh can be found in
Figure~\ref{fig:numerical-solutions}.

In a second series of experiments,  we study the performance of our
CutDG method in the presence of a sharp internal layer. For brevity,
we only report here in detail the results for sphere geometry as the
torus example produced very similar results.
First, we consider the case $\epsilon = 10^{-3}$ with a boundary layer width of 
$\sqrt{10^{-3}} \approx 0.0316$. Compared to the previous convergence test we now
consider an even larger number of successively finer meshes $\{\widetilde{\mcT}_k\}_{k=0}^9$
which guarantees that the internal layer is eventually resolved for the last 2 to 3 meshes.
This is also confirmed by the observed order of convergence displayed in 
Figure~\ref{fig:convergence_plots_sphere_eps_small} (top). Here, the convergence
rate behaves more erratic in the underresolved regime but eventually
approaches the theoretically predicted rates.

Finally, we consider the case $\epsilon = 10^{-6}$. Here, using only uniform mesh refinements,
we are not able to resolve the internal layer and the analytical solution behaves practically
like a discontinuous function from a numerical point of view.
This explains also the drastically reduced convergence rate reported in 
Figure~\ref{fig:convergence_plots_sphere_eps_small} (bottom).
Also, similar to standard fitted upwind DG methods, the numerical
solution exhibits the typical oscillatory behavior, known as the Gibbs phenomenon,
in the vicinity of the layer.  How to combine the proposed
stabilized CutDG framework with various shock capturing or limiter
techniques~\cite{Shu2009,Shu2016} to control spurious oscillation near
discontinuities will be part of our future research.

\begin{figure}[htb]
  \begin{subfigure}{1.0\textwidth}
  \begin{center}
    \includegraphics[width=0.47\textwidth,page=1]{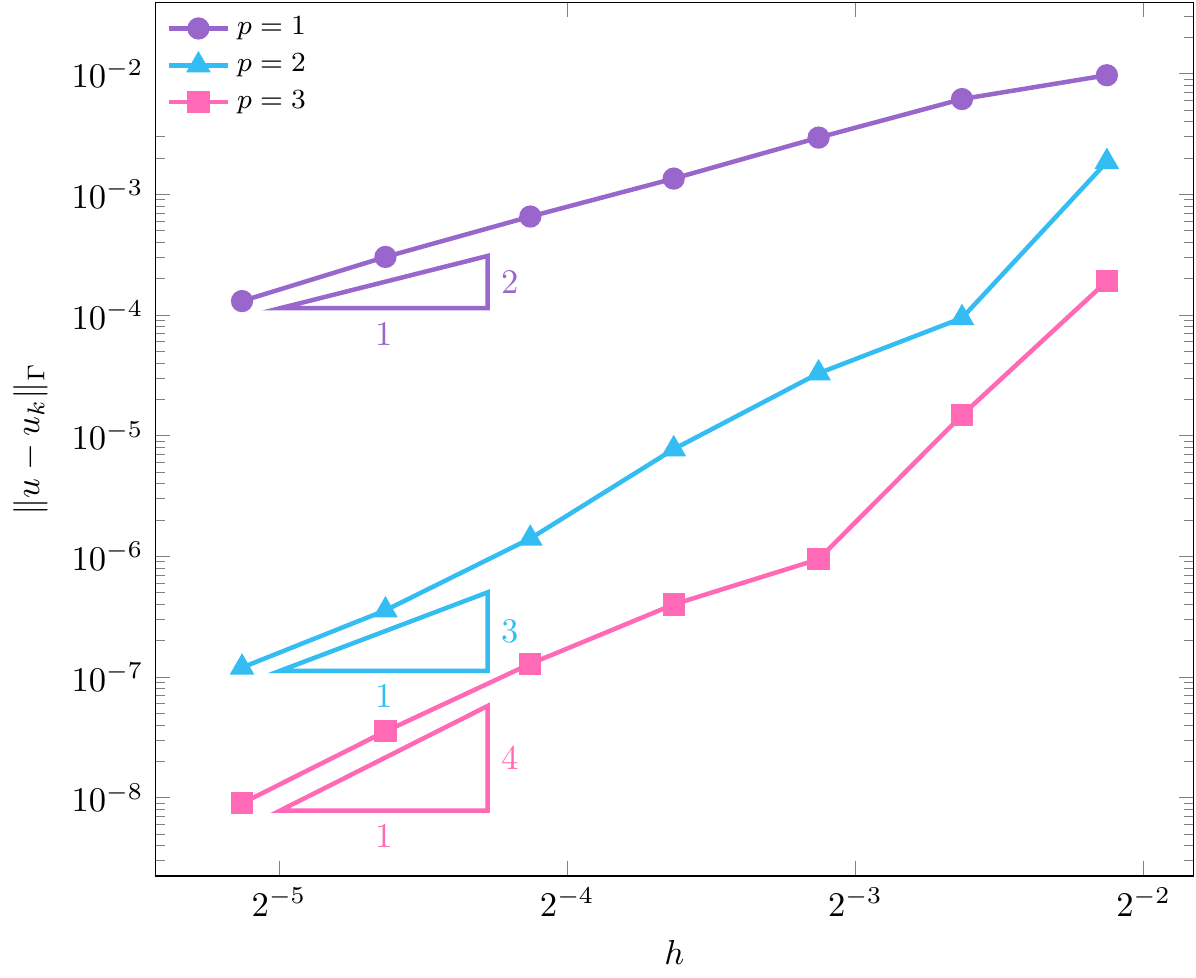}
    \hspace{0.01\textwidth}                                                 
    \includegraphics[width=0.47\textwidth,page=2]{convergence_plots.pdf}
  \end{center}
\end{subfigure}
  \begin{subfigure}{1.0\textwidth}
\vspace{2em}
  \begin{center}
    \includegraphics[width=0.47\textwidth,page=3]{convergence_plots.pdf}
    \hspace{0.01\textwidth}                                                 
    \includegraphics[width=0.47\textwidth,page=4]{convergence_plots.pdf}
  \end{center}
\end{subfigure}
\caption{Convergence rate plots for the sphere (top) and torus (bottom) test cases with 
$\epsilon = 1$. Both the
  $L^2$ (left) and streamline diffusion (right) error plots
  show optimal convergence rates.}
  \label{fig:convergence_plots_eps_large}
\end{figure}

\begin{figure}[htb]
    \begin{subfigure}{.49\columnwidth}
        \includegraphics[width=\columnwidth]{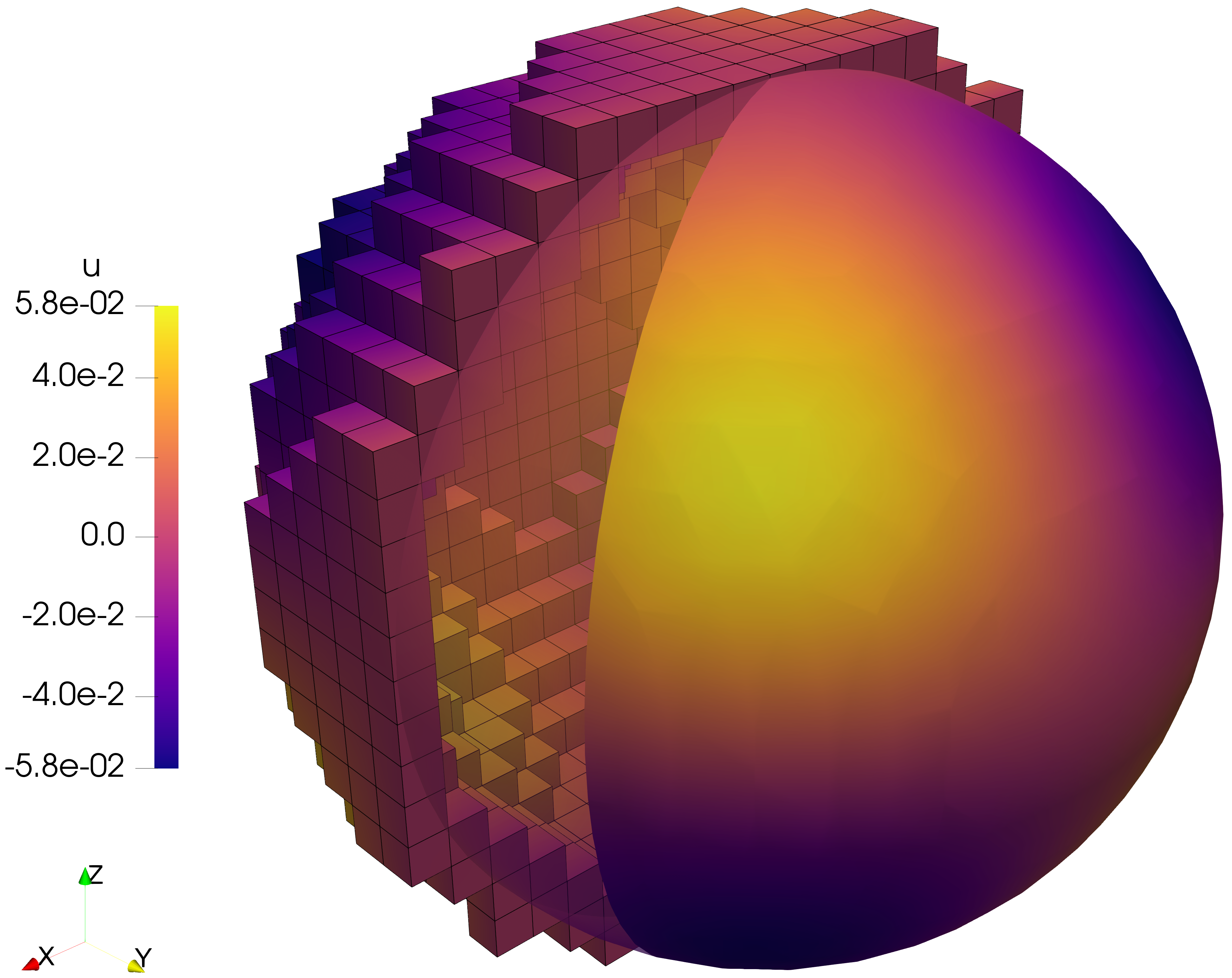}
    \end{subfigure}
    \begin{subfigure}{.49\columnwidth}
        \includegraphics[width=\columnwidth]{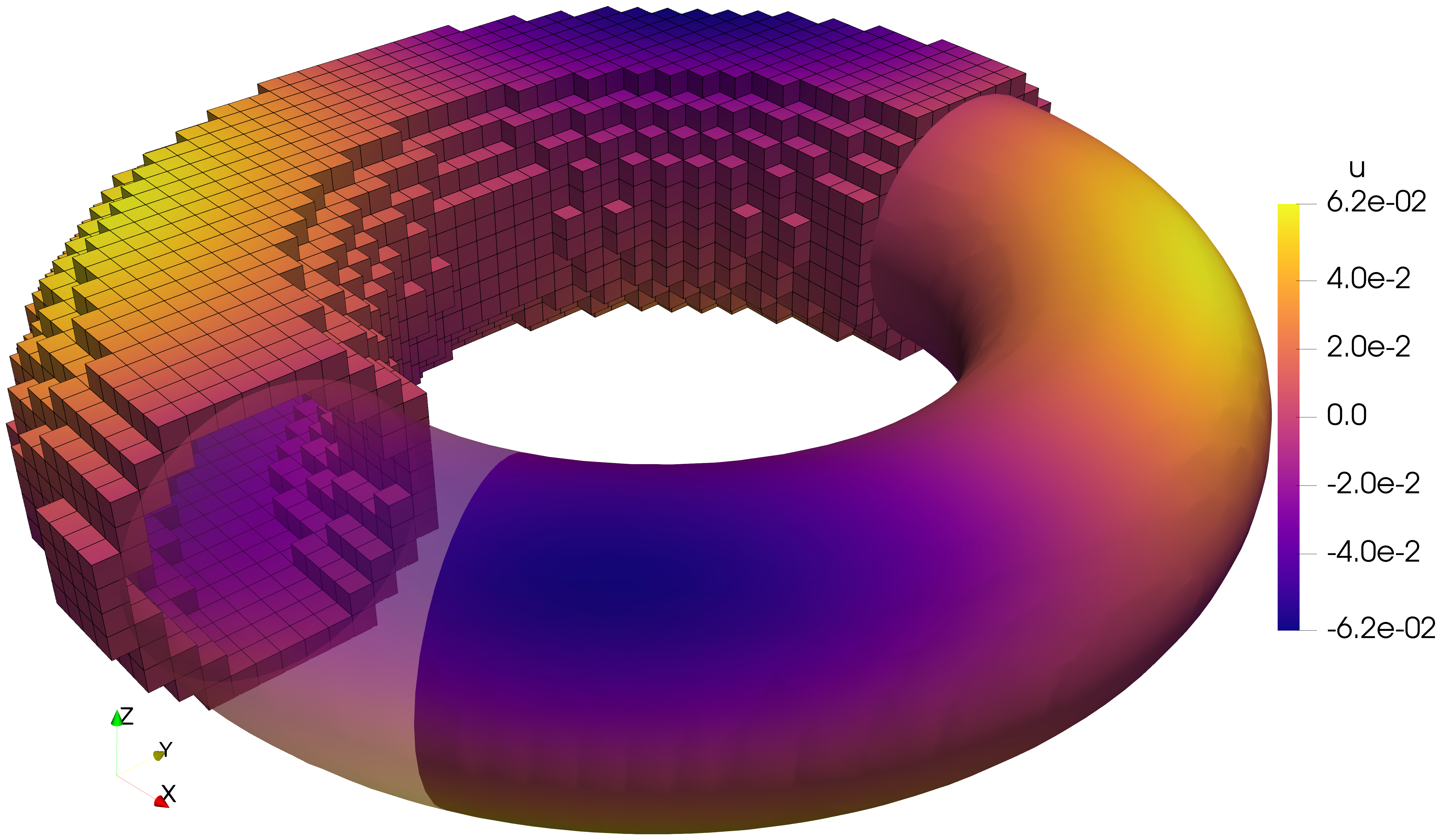}
    \end{subfigure}
    \caption{Numerical solutions of problem~\eqref{eq:experiments-data} with $\epsilon = 1$ computed on the unit sphere (left) and torus (right),
    together with part of the background mesh. 
    \label{fig:numerical-solutions}}
\end{figure}

\begin{figure}[htb]
    \begin{subfigure}{.49\columnwidth}
        \includegraphics[width=\columnwidth]{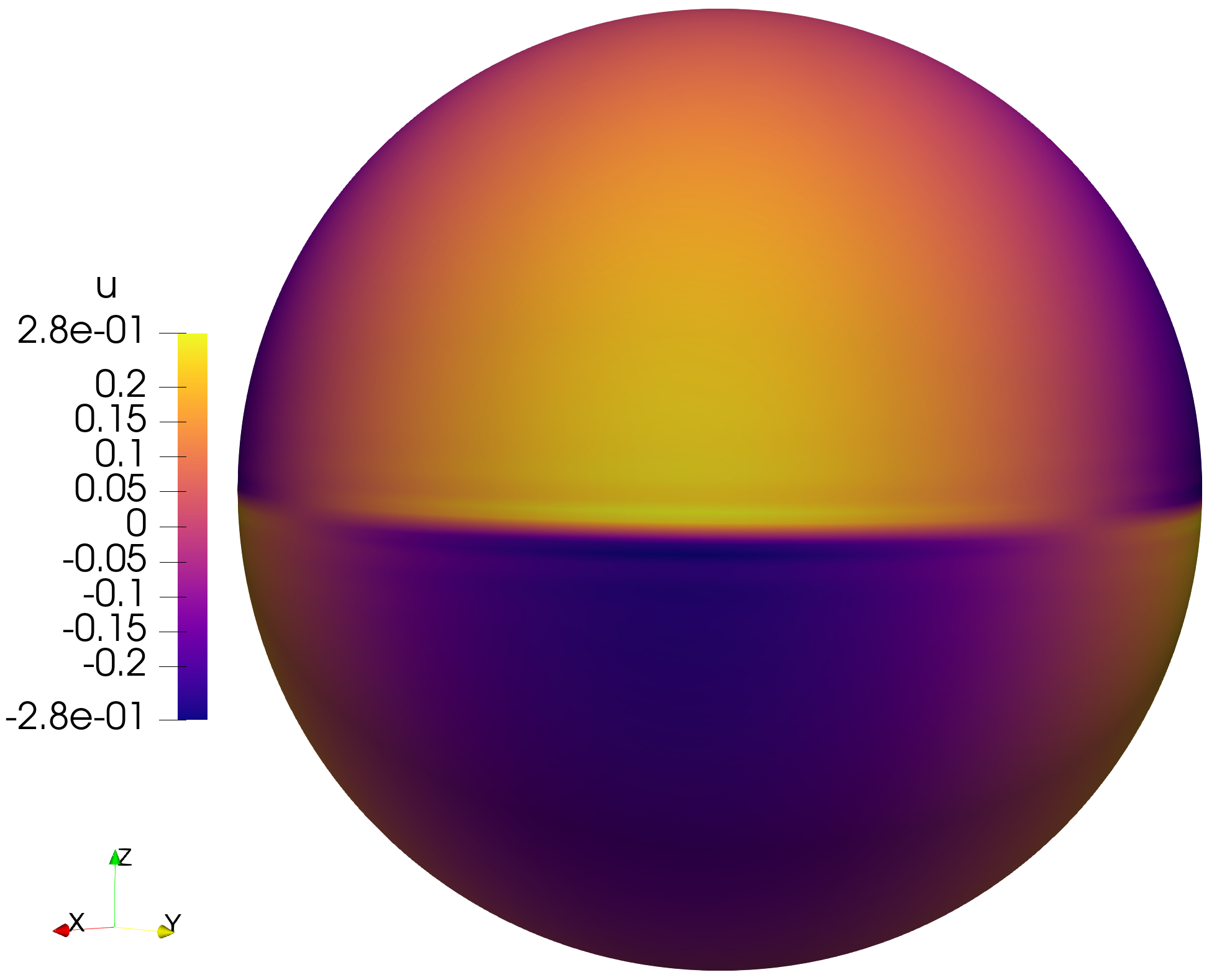}
    \end{subfigure}
    \begin{subfigure}{.49\columnwidth}
        \includegraphics[width=\columnwidth]{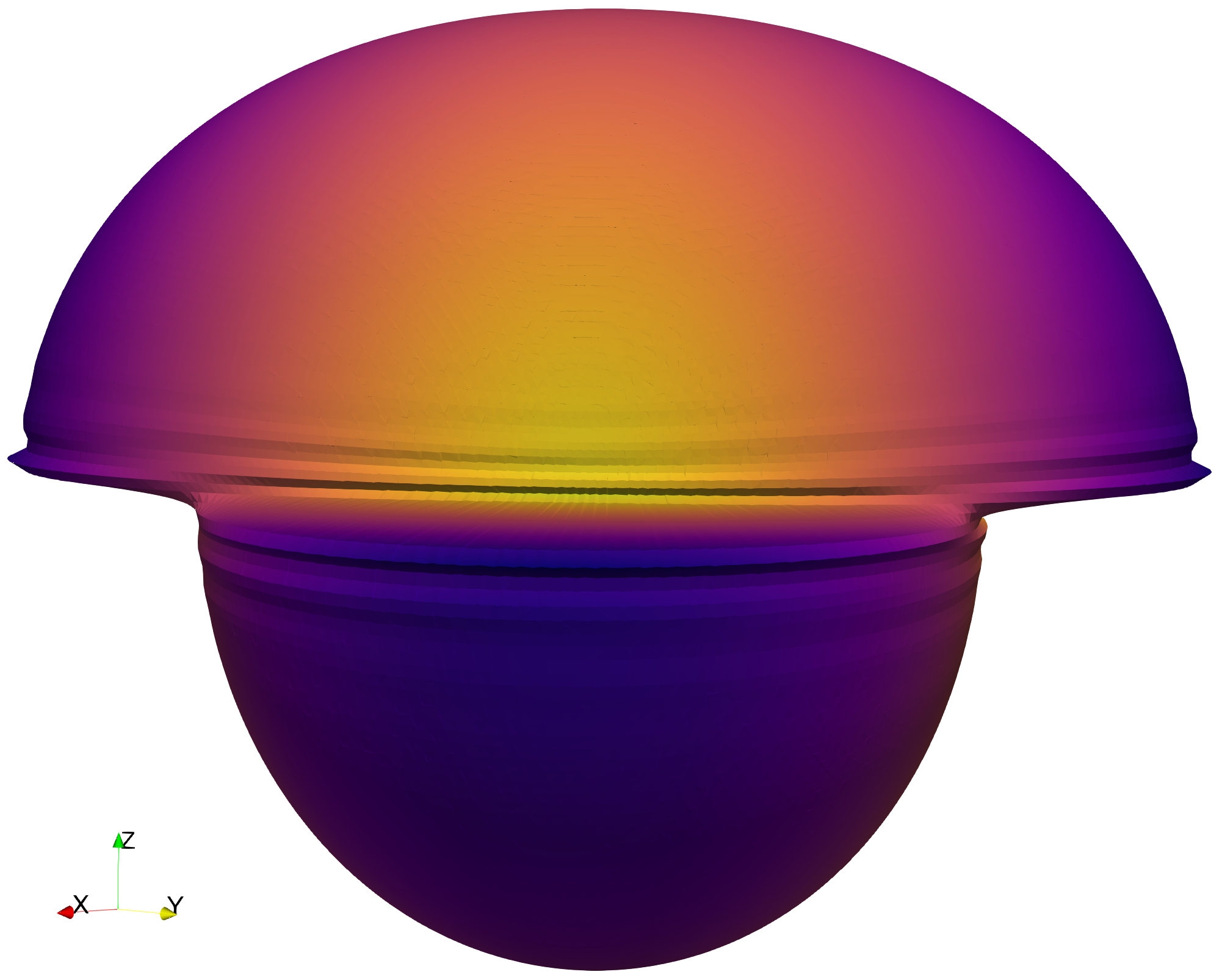}
    \end{subfigure}
    \caption{Numerical solution of problem~\eqref{eq:experiments-data} with $\epsilon = 10^{-6}$ computed on the unit sphere (left). Warping the surface $\Gamma$ in the normal direction using the solution $u$ exhibits the strong gradient in the characteristic
    layer around the equator as well as the localized Gibbs oscillations of the solution (right). 
    \label{fig:numerical-solutions-strong-grad}}
\end{figure}

\begin{figure}[htb]
  \begin{subfigure}{1.0\textwidth}
\vspace{2em}
  \begin{center}
    \includegraphics[width=0.47\textwidth,page=5]{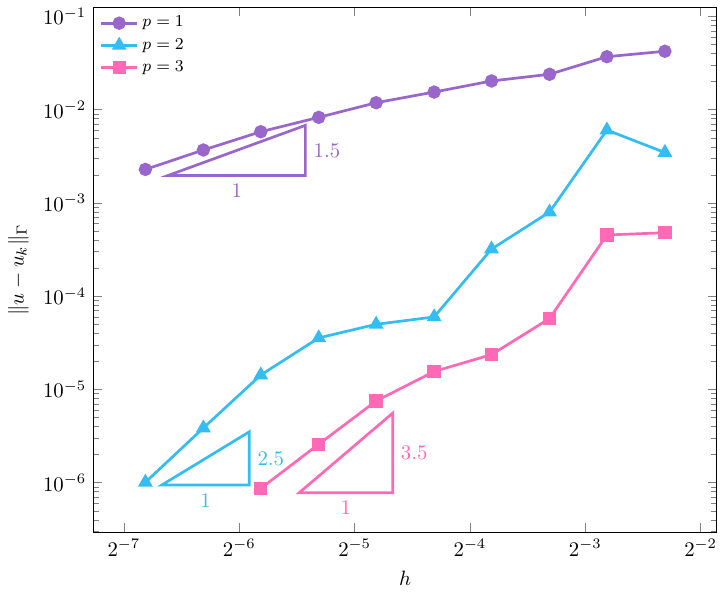}
    \hspace{0.01\textwidth}                                                 
    \includegraphics[width=0.47\textwidth,page=6]{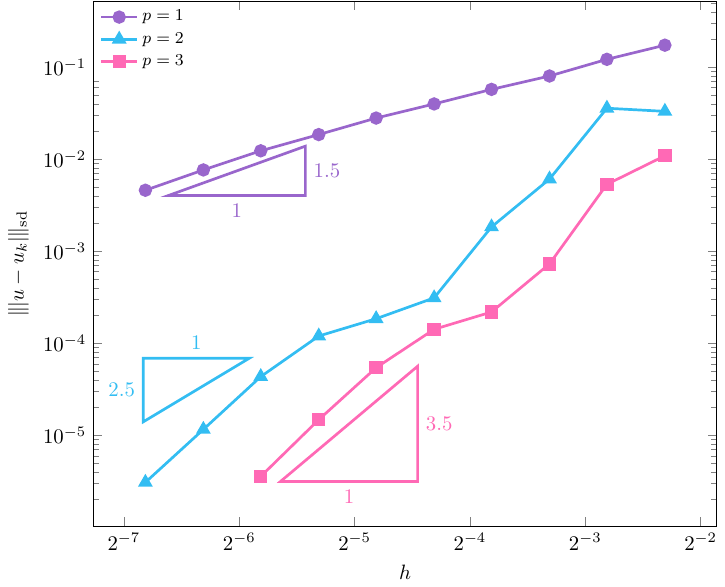}
  \end{center}
\end{subfigure}
  \begin{subfigure}{1.0\textwidth}
\vspace{2em}
  \begin{center}
    \includegraphics[width=0.47\textwidth,page=7]{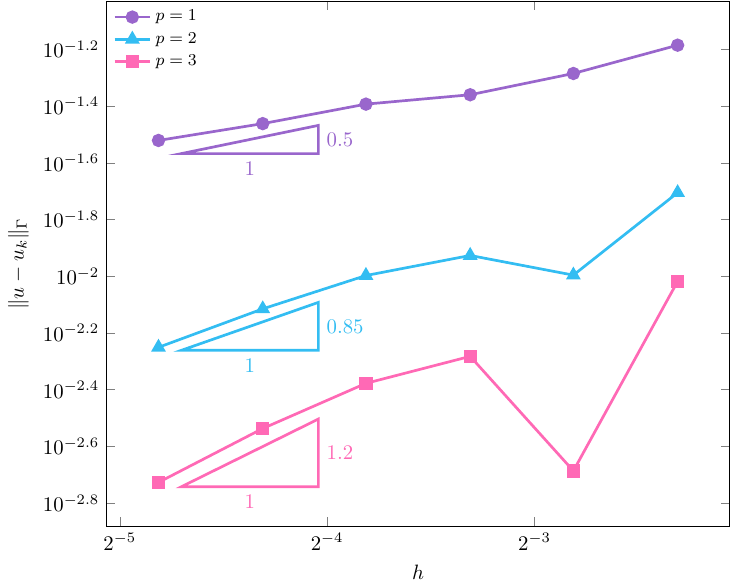}
    \hspace{0.01\textwidth}                                                 
    \includegraphics[width=0.47\textwidth,page=8]{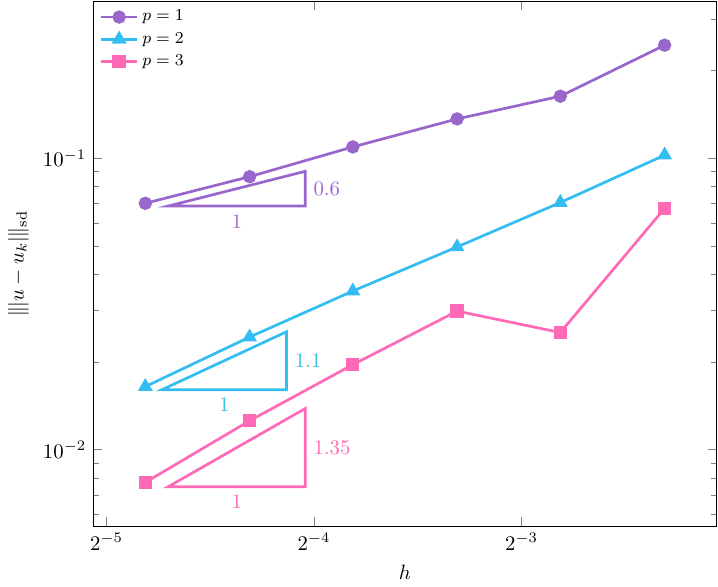}
  \end{center}
\end{subfigure}
  \caption{Convergence rates in the $L^2(\Omega)$ (left) and streamline
    diffusion (right) norms for the sphere example with $\epsilon = 10^{-3}$ (top) and $\epsilon = 10^{-6}$ (bottom).}
  \label{fig:convergence_plots_sphere_eps_small}
\end{figure}

\subsection{Condition number}
Next, we study the scaling of the condition number $\kappa(\mcA)$
of the system matrix $\mcA$ with respect to the mesh size $h$
for orders $k \in \{1,2,3\}$.
We consider the same experimental setup as for the sphere example.
To estimate the condition number for a given mesh $\mcT_l$ and
order $k$, we compute numerically the largest and smallest singular value of $\mcA$
using the SLEPc~\cite{HernandezRomanVidal2005},
an open-source library for the solution of large-scale sparse eigenvalue problems
which is closely integrated into deal.II.
The condition number as a function of mesh size is shown in Figure~\ref{fig:cond-vs-mesh-size},
for a few refinements and different orders.
Note that since the computation of singular values is computationally heavy and challenging, we were not able to perform equally many condition number calculations
for different orders.
As expected from Theorem~\ref{thm:condition-number},
we see that the condition number grows proportionally to $h^{-1}$.
\begin{figure}[htb]
    \centering
    \includegraphics[width=.35\paperwidth]{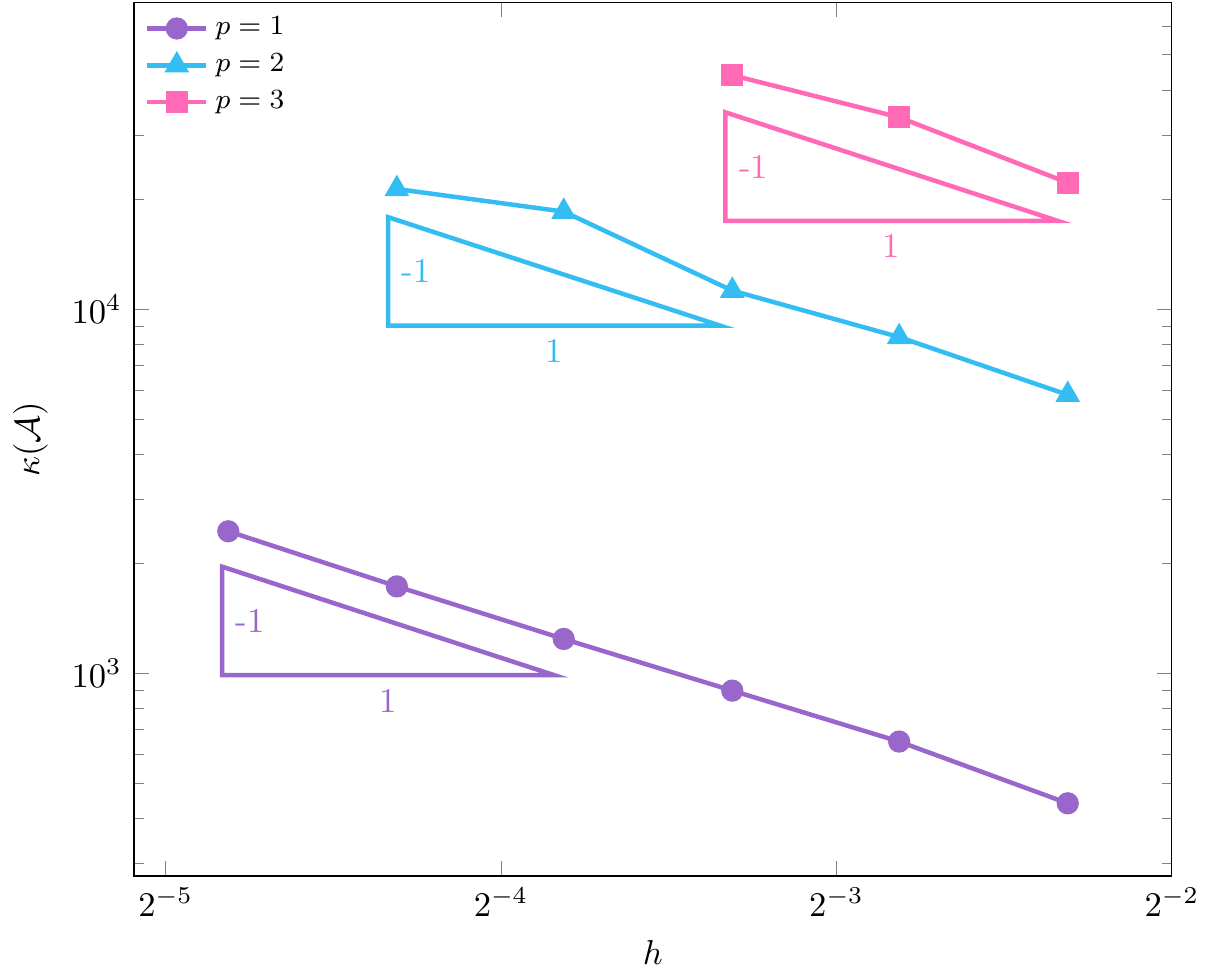}
    \caption{Condition number as a function of mesh size, for the test case with $\Gamma$ being a sphere.\label{fig:cond-vs-mesh-size}}
\end{figure}

\subsection{Geometrical robustness}
Finally, the last set of numerical experiments
is designed to test the geometrical robustness
of our proposed CutDG method and to highlight the importance
of the ghost penalty.

To test if the method yields robust approximation errors 
irrespective of the particular cut configuration,
we successively compute the numerical solution 
for the unit sphere test case from the previous section with $\epsilon = 1$
while shifting the background mesh by
\begin{equation}
    s_\delta = \delta \frac{h}{\sqrt{3}} (1, 1, 1), \quad \delta \in [0,1),
    \label{eq:mesh-perturbation}
\end{equation}
Here, $\delta$ is a parameter that quantifies the shift.
The problem is solved for $500$ uniformly spaced values of $\delta$ in
the interval $[0, 1)$ using polynomial order $k=2$.  In this interval,
the linear system has between $11232$ and $13014$ degrees of freedom.
For each sample, we compute both the discretization error measured in
the streamline-diffusion norm $\tn \cdot \tnsdh$ as well as the
condition number and plot them against $\delta$, see Figure~\ref{fig:mesh_perturbation}.
The resulting error sensitivities shown in Figure~\ref{fig:mesh_perturbation} (left)
include the results for the ``default'' parameters~\eqref{eq:penalty-parameters} as well as
the results when individual ghost penalty parameters are set to zero.
First, we see that when the penalty parameters have the default values from~\eqref{eq:penalty-parameters}
the error is independent of $\delta$.
When we set $\gamma_n = 0$, the error fluctuates rapidly
and increases by a factor higher than $10^4$ at some values of $\delta$.
If we instead set $\gamma_0 = 0$,
the error is almost the same as for the parameters in~\eqref{eq:penalty-parameters},
except for a few spikes, where the error increases significantly.
When setting $\gamma_1 = 0$,
the error is surprisingly robust and practically constant over $\delta$ but slightly higher than for the parameters in~\eqref{eq:penalty-parameters}.
It should be noted that, for each $\delta$,
the linear system was here solved with an iterative solver (bicgstab).
When setting some of the penalty parameters to zero,
the linear system might be singular.
Thus, what we have presented as the error in Figure~\ref{fig:mesh_perturbation}
is the solution from the iterative solver after a maximum of $10^4$ iterations,
even if the solver did not converge.
\begin{figure}[htb]
        \includegraphics[width=0.50\columnwidth,page=1]{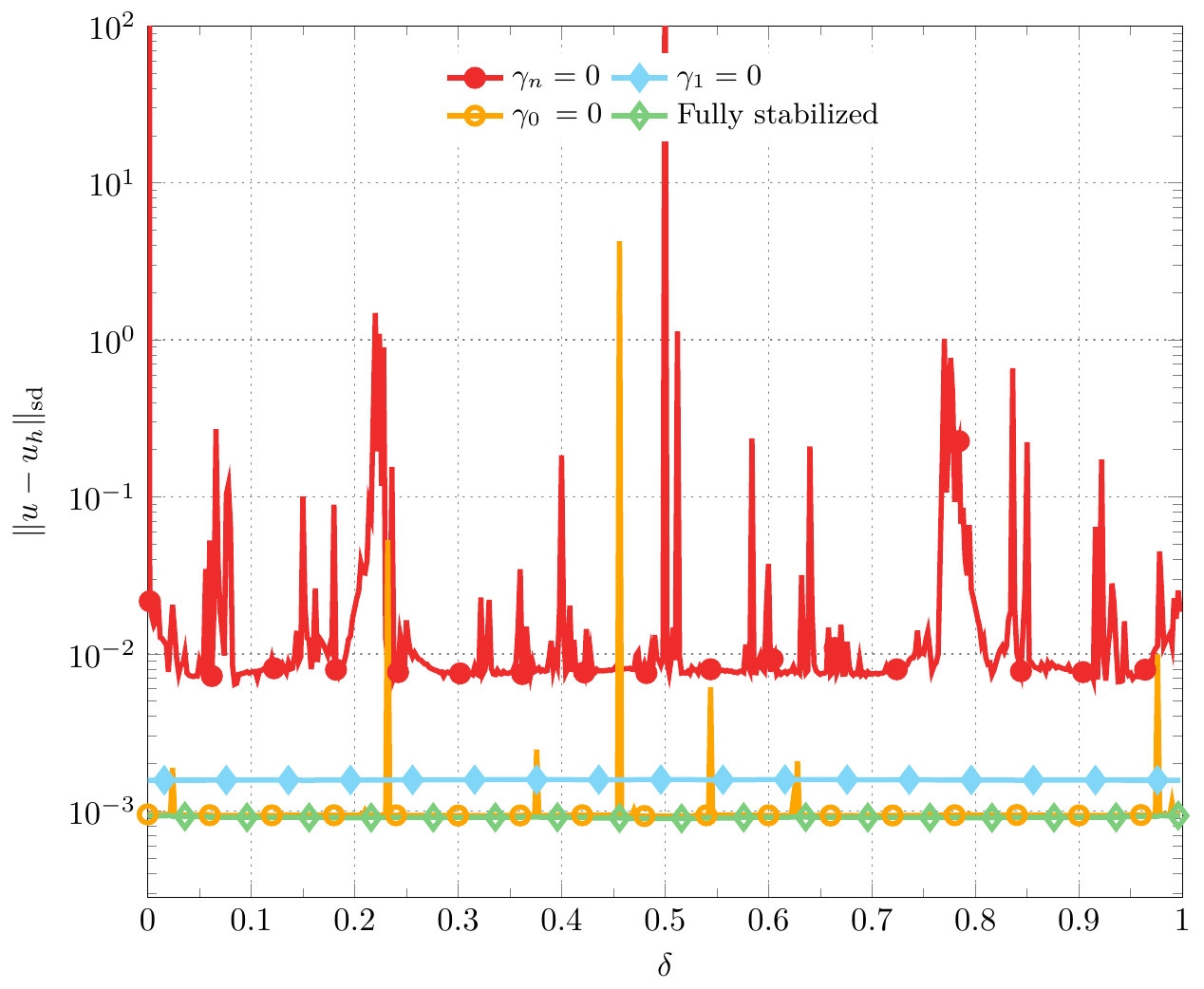}
        \includegraphics[width=0.50\columnwidth,page=2]{perturbation_experiments.pdf}
    \caption{Error in $\tn \cdot \tn_{\mrm{sd}}$-norm (left) and
    condition number (right) as a function of the mesh perturbation
    parameter $\delta$ (cf. \eqref{eq:mesh-perturbation})
    for different choices of stability parameters. \label{fig:mesh_perturbation}
    }
\end{figure}

As expected, the condition number is more mesh-dependent when we
significantly decrease the stabilization parameters.  When we set
$\gamma_n = 10^{-4}$ the condition number increases by almost 3 orders
of magnitude, but is still almost constant.  If we instead set
$\gamma_0 =10^{-8}$, we see that the condition number becomes huge and
also oscillates rapidly.  Here, some values of the condition number
are missing.  The reason is that the matrix, $\mathcal{A}$, is so
ill-conditioned that the used singular value solver did not converge
when solving for the smallest singular value.  Of the various penalty
parameters, $\gamma_1$ appears to be the one that has the smallest
effect.  When we set $\gamma_1 = 0$ we see that the condition number
becomes slightly more mesh-dependent compared to the parameters in
\eqref{eq:penalty-parameters}, but the variation is very slight.

\section{Conclusion and outlook}
In this paper, we proposed a novel cut discontinuous Galerkin method
for stationary advection-reaction problems on surfaces.  Our main goal
was to generalize the classical upwind-flux DG formulation to the
setting of embedded surfaces by extending ideas from the stabilized,
continuous Galerkin-based CutFEM
framework~\cite{GrandeLehrenfeldReusken2018,BuermanHansboLarsonEtAl2018}
for surface PDEs.
We carefully designed suitable stabilization forms for higher-order DG methods which allowed
us to establish geometrically robust stability, a priori error, and condition
number estimates by using enhanced $L^2$ and streamline-diffusion type norms.
Moreover, the presented stabilization approach allows for a relatively
easy extension of existing fitted discontinuous Galerkin
software to handle unfitted geometries. Implementation of
the stabilization operator~\eqref{eq:sh-def} should be straight-forward 
in most DG software frameworks, and thus
only additional quadrature routines such
as~\cite{MuellerKummerOberlack2013,Saye2015,Lehrenfeld2016,FriesOmerovic2015,FriesOmerovicSchoellhammerEtAl2017}
are needed to handle the numerical integration on cut
geometries.

In this work, we focused on the prototype problems~\eqref{eq:continuous_problem}
to lay out the main ideas in the simplest possible setting,
but our method can be readily employed 
in more complex simulation scenarios,
including advection-dominated advection-diffusion-reaction problems on surfaces
when combined with~\cite{BurmanHansboLarsonEtAl2016a},
or for corresponding mixed-dimensional problems
in combination with ~\cite{Massing2017,GuerkanMassing2019,GuerkanStickoMassing2020}. 
In~\cite{GuerkanStickoMassing2020}
we already outlined relevant extensions and research directions 
for the proposed stabilized CutDG formulation for advection-dominated bulk problems.
In particular, we demonstrated how the stabilization
approach can be combined with explicit Runge--Kutta methods to solve
the time-dependent advection-reaction problem under a standard
hyperbolic CFL condition. The research directions and method extensions
from~\cite{GuerkanStickoMassing2020}
are equally applicable to CutDG formulation in this work.
It is part of our ongoing research to combine the presented
stabilized CutDG framework with the general symmetric stabilization
approach proposed in~\cite{BurmanErnFernandez2010}
to devise an explicit Runge--Kutta method for first-order Friedrichs-type
operators covering advection-reaction problems
as well as linear wave propagation phenomena.

Moreover, for the numerical discretization of nonlinear scalar
hyperbolic conservation laws on surfaces,
a major research question is to understand how
our proposed CutDG stabilization can be combined
with the discontinuous Galerkin Runge--Kutta methods originally
developed in~\cite{CockburnShu1991,CockburnShu1989,CockburnLinShu1989,CockburnHouShu1990}. 
To maintain properties such as local conservation, monotonicity,
total variation diminishing (TVD) stability 
often required from numerical methods for hyperbolic conservation
laws, modifications of the proposed stabilizations need to be
developed. Here, it would be interesting to investigate whether and
how the approaches developed in
\cite{MayStreitbuerger2022,EngwerMayNuessingEtAl2020} can be carried
over to the setting of embedded surfaces.

\section*{Acknowledgments}
The authors gratefully acknowledge financial support 
from the Swedish eSSENCE program of e-Science
and from the Swedish Research Council under Starting Grant 2017-05038.
We also wish to thank the anonymous reviewers
for their valuable comments which helped us to improve the quality of this paper.

\appendix

\section{Proofs of some geometric estimates}
\label{sec:appendix}

\begin{proof}[Lemma~\ref{lem:coeff-est}]
To establish~\eqref{eq:bh_assumption}, we simply combine estimates~\eqref{eq:B_estimates}, \eqref{eq:B_det_estimates}, and~\ref{eq:bh_assumption_prel} to obtain 
\begin{align}
    \||B| B^{-1} b^e - b_h\|_{L^\infty(\mcK_h)} 
    &\lesssim 
    \||B| B^{-1} b^e - B^{-1} b^e\|_{L^\infty(\mcK_h)} + \|B^{-1} b^e - b_h\|_{L^\infty(\mcK_h)}
    \\
    &\lesssim \||B| - 1\|_{L^\infty(\mcK_h)}\|B^{-1} b^e\|_{L^\infty(\mcK_h)} + \|B^{-1} b^e - b_h\|_{L^\infty(\mcK_h)}
    \\
    &\lesssim h^{k_g+1} + \|B^{-1} b^e - b_h\|_{L^\infty(\mcK_h)}
    \\
    &\lesssim h^{k_g+1} + \|(B^{-1} -\Psh \Ps ) b^e\|_{L^\infty(\mcK_h)} + \|\Psh \Ps b^e - b_h\|_{L^\infty(\mcK_h)}\\
    &\lesssim h^{k_g+1} + \|\Psh b^e - b_h\|_{L^\infty(\mcK_h)}.
\end{align}
Inequalities~\eqref{eq:ch_assumption} and~\eqref{eq:fh_assumption}
can be proved similarly.
\qed\end{proof}

\begin{proof}[Lemma~\ref{lem:est-bh-jump}]
    We start the proof by noting that in contrast to their discrete counterparts $n_{E}^{\pm}$, 
    the two co-normal fields $n_{E^l}^{\pm}(x) \in T_x \Gamma$ associated with the \emph{lifted}
    edge $E^l$ are in fact co-planar and satisfy $n_{E^l}^+ = - n_{E^l}^-$
    and hence $[b^e;n_{E^l}^e] = 0$.
    Thus 
    \begin{align}
        \| [b_h; n_E] \|_{L^{\infty}(\mcE_h)} 
        \leqslant
        \| [b_h - b^e; n_E] \|_{L^{\infty}(\mcE_h)} 
        +\| [b^e; n_E-n_{E_l}^e] \|_{L^{\infty}(\mcE_h)} 
        = I + II.
    \end{align}
    To estimate $I$, simply observe that 
    $b^{e,\pm} \cdot n_E^{\pm} =  (\Psh b^{e})^{\pm} \cdot n_E^{\pm}$,
    which thanks to assumption~\eqref{eq:bh_assumption_prel} implies 
    that 
    \begin{align}
       I 
        &\lesssim
        \| (b_h - \Psh b^e)^+ \cdot n_E^+] \|_{L^{\infty}(\mcE_h)} 
        + \| (b_h - \Psh b^e)^- \cdot n_E^-] \|_{L^{\infty}(\mcE_h)} 
        \lesssim C_b h^{k_g+1}.
    \end{align}
    Next, using the fact that 
    $b^e \cdot n_{E}^{\pm} 
    = \Ps^e b^e \cdot n_{E}^{\pm} 
    = b^e \cdot \Ps^e n_{E}^{\pm} 
    $ thanks to the self-adjointness of $\Ps$,
    the remaining term $II$ can be bound by
    \begin{align}
       II
       &\leqslant
       \|b^e\cdot (\Ps^e  n_E^+
       - n_{E^l}^{+,e})\|_{L^{\infty}(\Gamma_h)}
       + \| b^e\cdot
       (\Ps^e n_E^-
       - n_{E^l}^{-,e}) \|_{L^{\infty}(\Gamma_h)} =
       II_a + II_b.
    \end{align}
    Clearly, it is sufficient to provide an estimate
    for $II_a \leqslant b_c \|\Ps^e  n_E^+ - n_{E^l}^{+,e}\|_{L^{\infty}(\Gamma_h)}$
    since $II_b$ can be handled in the exact same manner.
    We introduce a moving orthonormal basis $t_1(x), \ldots, t_{d-1}(x) \in T_x E$
    such that $\{t_1, \ldots, t_{d-1}, n_{E}, \nsh\}$ is a positively oriented orthonormal basis of $\RR^{d+1}$.
    Then $-n_E = t_1 \wedge t_2 \wedge \ldots \wedge t_{d-1} \wedge \nsh$.
    Here and in the following, we omit the superscripts~$^{\pm}$ and~$^e$ to ease the notation.
    After a rigid motion, we can safely assume that 
    $\{t_1, \ldots, t_{d-1}, n_{E}, \nsh\} = \{e_1, \ldots, e_{d-1}, e_d, e_{d+1}\}$.
    Note that in this orthonormal basis, we have that
    $\ns^i = \ns \cdot e_i = (\ns - \nsh) \cdot e_i = \mcO(h^{k_g})$
    for $i=1,\ldots, d$ and hence $\ns^{d+1} = \sqrt{1+\mcO(h^{2k_g})} = 1 + \mcO(h^{2k_g})$.
    Now we expand $\Ps n_E=\Ps e_{d}$ in the chosen orthonormal basis leading to
       \begin{align}
        \label{eq:Ps_ed_expansion}
        \Ps e_{d} \cdot e_i
        = \delta_{d,i} -\ns^d \ns^i 
        = \begin{cases}
            \mcO(h^{2 k_g}), &i = 1,\ldots, d-1
            \\
            1 + \mcO(h^{2 k_g}), & i = d
            \\
            - \ns^d \ns^{d+1},   & i = d+1.
        \end{cases}
       \end{align}
    Next, using the differential $Dp = \Ps(\Id - \rho \mcH)\Psh$ of the closest point projection $p$, we define 
    \begin{align}
    - \widetilde{n}_{E^l} =  Dp e_1 \wedge Dp e_2 \wedge Dp e_{d-1} \wedge \ns,
    \end{align}
    which is a \emph{non-normalized}, outward pointing co-normal field on the lifted edge $E^l$; that is,
    $n_{E^l} = \lambda \widetilde{n}_{E^l}$ for some $\lambda > 0$.
    Recalling the general definition of the outer product, we see that
    \begin{align}
    -\widetilde{n}_{E^l} \cdot e_i 
    &= 
    \det(Dp e_1, Dp e_2, Dp e_{d-1}, \ns, e_i)
    \\
    &= 
    \det(\Ps e_1, \Ps e_2, \Ps e_{d-1}, \ns, e_i) + \mcO(h^{k_g+1})
    \\
    &=
    \det(e_1, e_2, e_{d-1}, \ns, e_i) + \mcO(h^{k_g+1})
    \\
    &=
    \begin{cases}
        0 + \mcO(h^{k_g+1})  &i=1,\ldots,d-1 \\
        -\ns^{d+1} + \mcO(h^{k_g+1}) = -1 + \mcO(h^{k_g+1}) &i=d \\
        \ns^{d}  + \mcO(h^{k_g+1}) &i=d+1.
    \end{cases}
    \end{align}
    As all coefficients scale like at least $\mcO(h^{k_g})$ except for $i=d$, we see that
    $\| \widetilde{n}_{E^l}\|_{\RR^{d+1}} = 1 + \mcO(h^{2k_g})$, hence 
    $ \lambda =\| \widetilde{n}_{E^l}\|_{\RR^{d+1}}^{-1} = 1 + \mcO(h^{2k_g})$ and
    consequently we obtain the following estimates for the coefficients of $n_{E^l}$
    with respect to the orthonormal base $\{e_1\ldots,e_{d+1}\}$,
    \begin{align}
        \label{eq:n_El_expansion}
     -n_{E^l} 
     = 
     -\lambda\widetilde{n}_{E^{l}}   
     =-(1 + \mcO(h^{2k_g}))\widetilde{n}_{E^{l}}   
    &= \begin{cases}
        0 + \mcO(h^{k_g+1})  &i=1,\ldots,d-1 \\
        -1 + \mcO(h^{k_g+1}) &i=d, \\
        \ns^{d}  + \mcO(h^{k_g+1}) &i=d+1.
    \end{cases}
    \end{align}
    As a result, comparing~\eqref{eq:Ps_ed_expansion} and~\eqref{eq:n_El_expansion} yields
    \begin{align}
        (\Ps e_{d} - n_{E^l}^e)  \cdot e_i
        = \begin{cases}
            \mcO(h^{k_g+1}) &i=1,\ldots,d
            \\
            \ns^d(1-\ns^{d+1}) + \mcO(h^{k_g+1}) =  \mcO(h^{k_g+1}), &i = d+1,
        \end{cases}
    \end{align}
   which immediately implies that 
    $II_a \leqslant b_c h^{k_g+1}$.
   This concludes the proof.
\qed\end{proof}


\bibliographystyle{elsarticle-num-names.bst}
\bibliography{bibliography-zotero}

\end{document}